\documentclass[12pt, a4paper]{amsart}

\usepackage{amsmath, amsfonts, amssymb, amsthm}
\usepackage[margin=1.25in, top=1.5in, bottom=1in]{geometry}
\usepackage[all]{xy}
\usepackage{tikz}

\usepackage{tikz-cd}
\usepackage{setspace}

\newcommand{\RR}{\mathbb{R}}
\newcommand{\CC}{\mathbb{C}}

\newcommand{\NN}{\mathbb{N}}
\newcommand{\ESS}{\mathbb{S}}
\newcommand{\T}{\mathbb{T}}
\newcommand{\ZZ}{\mathbb{Z}}

\newcommand{\Aut}{\operatorname{Aut}}

\newcommand{\TT}{\mathcal{T}}

\newcommand{\id}{\operatorname{id}}

\newcommand{\sub}{\operatorname{sub}}

\newcommand{\clsp}{\overline{\operatorname{span}}}

\allowdisplaybreaks
\newtheorem{thm}{Theorem}
\newtheorem{lemma}[thm]{Lemma}

\newtheorem{cor}[thm]{Corollary}
\newtheorem{prop}[thm]{Proposition}

\theoremstyle{definition}

\newtheorem{example}[thm]{Example}
\newtheorem{remark}[thm]{Remark}

\allowdisplaybreaks

\numberwithin{equation}{section}
\numberwithin{thm}{section}

\title[KMS states]{\boldmath{Equilibrium states on higher-rank Toeplitz noncommutative solenoids}}

\author{Zahra Afsar}

\author{Astrid an Huef}
\author{Iain Raeburn}

\author{Aidan Sims}
\address{Zahra Afsar, School of Mathematics and Statistics, University of Sydney, NSW 2006, Australia}
\address{Astrid an Huef and Iain Raeburn, School of Mathematics and Statistics, Victoria University of Wellington, PO Box 600, Wellington 6140, New Zealand}
\address{Aidan Sims, School of Mathematics and Applied Statistics, University of Wollongong, NSW 2522, Australia}
\date{\today}

\thanks{This research was supported by the Australian Research Council and the Marsden Fund of the Royal Society of New Zealand.}

\begin{document}

\begin{abstract}
We consider a family of higher-dimensional noncommutative tori, which are twisted
analogues of the algebras of continuous functions on ordinary tori, and their Toeplitz
extensions. Just as solenoids are inverse limits of tori, our Toeplitz noncommutative
solenoids are direct limits of the Toeplitz extensions of noncommutative tori. We
consider natural dynamics on these Toeplitz algebras, and compute the equilibrium states
for these dynamics. We find a large simplex of equilibrium states at each positive
inverse temperature, parametrised by the probability measures on an (ordinary) solenoid.
\end{abstract}


\maketitle

\section{Introduction}

Classical solenoids are inverse limits of tori. There are noncommutative analogues of
tori, which are the twisted group algebras  $C^*(\ZZ^n,\sigma)$ of the abelian group
$\ZZ^n$. For $n=2$, these are the rotation algebras $A_\theta$ generated by two unitaries
$U,V$ satisfying the commutation relation $UV=e^{2\pi i\theta}VU$. When $\theta$ is
irrational, these are simple $C^*$-algebras, and have been extensively studied (see, for
example, \cite[Chapter~VI]{D}). For $\theta=0$, we recover the commutative algebra
$C(\T^2)$, and hence the $A_\theta$ are also known as ``noncommutative tori.'' In
\cite{LP}, Latr{\'e}moli{\`e}re and Packer studied a family of noncommutative solenoids
that are direct limits of noncommutative tori. (The connection is that the commutative
algebra of continuous functions on a solenoid is the direct limit of the algebras of
continuous functions on the approximating tori.)

Following surprising results about phase transitions for the KMS states of the Toeplitz
algebras of the $ax+b$-semigroup of the natural numbers \cite{LR, LN2}, many authors have studied the KMS structure of Toeplitz extensions in other settings. Typically, these Toeplitz extensions exhibit more interesting KMS structure. This recent work has covered Toeplitz algebras of directed graphs and their higher-rank analogues \cite{aHLRS1, aHLRS,
C1, FaHR, C2} (after earlier work in \cite{EL}), Toeplitz algebras arising in number theory \cite{CDL}, the Nica-Toeplitz extensions of Cuntz-Pimsner algebras \cite{LN2, K1,K2, ABLS, ALN}, and Toeplitz algebras associated to self-similar actions \cite{LRRW, LRRW2}. In \cite{BHS}, Brownlowe, Hawkins and Sims described Toeplitz extensions of the noncommutative solenoids from \cite{LP}, and considered a natural dynamics on this extension. They showed that for each inverse temperature $\beta>0$, the KMS$_\beta$ states are parametrised by the probability measures on a commutative solenoid which is the inverse limit of $1$-dimensional tori
\cite[Theorem~6.6]{BHS}.

Here we consider a family of higher-rank noncommutative solenoids and their Toeplitz
extensions. As for the algebras of higher-rank graphs \cite{aHLRS}, there is an obvious
gauge action of a torus $\T^d$ on these algebras, but to get a dynamics one has to choose
an embedding of $\RR$ in the torus. We fix $r\in [0,\infty)^d$, giving an embedding
$t\mapsto e^{itr}$ of $\RR$ in $\T^d$, and compose with the gauge action to get a
dynamics $\alpha^r$.

The building blocks in \cite{BHS} are Toeplitz noncommutative tori in which one generator
$U$ is unitary, the other $V$ is an isometry, the relation is still given by $UV=e^{2\pi
i\theta}VU$, and the dynamics fixes $U$. Here we fix $d,k\in \NN$. Our blocks $B_\theta$
are Toeplitz noncommutative tori generated by a unitary representation $U$ of $\ZZ^d$ and
a Nica-covariant isometric representation $V$ of $\NN^k$, and the commutation relation is
given by  $U_nV_p=e^{2\pi ip^T\theta n}V_pU_n$ for a fixed $k\times d$ matrix $\theta$
with entries in $[0,\infty)$. Then the dynamics $\alpha^r$ is given by a vector $r\in
(0,\infty)^k$; it fixes the unitaries $U_n$, and multiplies $V_p$ by $e^{itp^Tr}$.

We begin  by describing the direct system of Toeplitz noncommutative tori whose limit is
the Toeplitz noncommutative solenoid of the title. Everything is defined in terms of
presentations of the blocks: building the connecting maps is in particular quite
complicated, and requires us to be careful with the notation, which we try to keep
consistent throughout the paper. We then discuss the dynamics, which is again defined
using actions on the individual blocks. Then, remarkably, we have a presentation of the
direct limit which allows us to state our main result as Theorem~\ref{newbigthm}. This
gives a satisfyingly explicit description of the KMS$_\beta$ states in terms of measures
on a commutative solenoid of the form $\varprojlim \T^d$. This concrete description is
new even in the case $k=d=1$ studied in \cite{BHS}.

The first step in the proof of our theorem is an analysis of the KMS states of a building
block $B_\theta$, which we do in \S\ref{sec:NCT-new}. The description in
Proposition~\ref{prp:KMS char} looks rather like the descriptions of KMS states on graph
algebras in \cite[Theorem~3.1]{aHLRS1} and \cite[Theorem~6.1]{aHLRS}, and on algebras
associated to local homeomorphisms in \cite[Theorem~5.1]{AaHR}: we find a subinvariance
relation which identifies the measures on the torus associated to KMS states, and then
describe the solutions of that relation in terms of a concrete simplex of measures.

In the next section (\S\ref{ctssubinv}), we show how the subinvariance relations for the
building blocks combine to give one continuously parametrised subinvariance relation for
the direct limit (Theorem~\ref{subinv}).  We then describe the solutions to this new
subinvariance relation in Theorem~\ref{thm:sub-inv}, which is the key technical result in
the paper. This solution is very concrete, involving a formula which is reminiscent of a
multi-variable Laplace transform, and is much more direct than the \emph{ad hoc} approach
used in \cite{BHS}.

In the last section, we give a concrete description of the isomorphism
$\mu\mapsto\psi_\mu$  of the simplex $P(\varprojlim \T^d)$ of probability measures on the
solenoid onto the simplex of KMS$_\beta$ states on the Toeplitz noncommutative torus.
Then by evaluating these KMS states on generators, we arrive at the explicit values
described in Theorem~\ref{newbigthm}.

\section{Toeplitz noncommutative solenoids}\label{Tncs}

We define a  Toeplitz noncommutative solenoid as the direct limit of a sequence of
blocks, which we call Toeplitz noncommutative tori. So we begin by looking at these
blocks. In the course of this section we will introduce notation which will be used
throughout the paper.

 First we fix positive integers $d$ and $k$.  We write $A^T$ for the transpose of a matrix $A$. We view elements of $\RR^k$ as column vectors, and write the inner product of $n,p\in \RR^k$ in matrix notation as $p^Tn$. We use similar conventions for $\RR^d$.

The pair $(\ZZ^k,\NN^k)$ is a quasi-lattice ordered group in the sense of Nica \cite{N}.
Indeed, for every $p,q\in \NN^k$, the element $p\vee q$ defined pointwise by
\[
(p\vee q)_j=\max\{p_j,q_j\}\quad \text{for $1\leq j\leq k$}
\]
is a least upper bound for $p$ and $q$, so it is lattice-ordered. An isometric
representation $V:\NN^k\to B(H)$ is \emph{Nica-covariant} if it satisfies
\[
V_pV_p^*V_qV_q^*=V_{p\vee q}V_{p\vee q}^*\quad\text{for all $p,q\in \NN^k$,}
\]
or equivalently \cite[(1.4)]{LR1} if
\[
V_p^*V_q=V_{(p\vee q)-p}V^*_{(p\vee q)-q} \quad\text{for all $p,q\in \NN^k$.}
\]

For $\theta\in M_{k,d}([0,\infty))$, we consider the  universal $C^*$-algebra $B_\theta$
generated by a unitary representation $U$ of $\ZZ^d$ and a Nica-covariant isometric
representation $V$ of $\NN^k$ such that
\begin{equation}\label{relation:V-U}
U_n V_p= e^{2\pi ip^T \theta n}  V_pU_n\quad\text{for $p,q\in\NN^k$ and $n\in \ZZ^d$.}
\end{equation}
We then have also
\begin{equation}\label{relation:V*-U}
U_nV^*_p= (V_pU_{-n})^* = \big(e^{-2\pi i p^T \theta (-n)} U_{-n}V_p\big)^* = e^{- 2\pi ip^T \theta n}V^*_pU_n .
\end{equation}
Direct calculation shows that for $p,q,p',q' \in \NN^k$ and $n, n' \in \ZZ^d$, we have
\begin{align*}
V_p U_n V^*_q &V_{p'} U_{n'} V^*_{q'}
    = V_p U_n V_{(q \vee p')-q} V^*_{(q \vee p')-p'} U_{n'} V^*_{q'}\\
    &= e^{2\pi i ((q \vee p'-q)^T\theta n+(q \vee p'-p')^T\theta n')}V_{p + (q \vee p') - q} U_{n+n'} V^*_{q' + (q \vee p') - p'},
\end{align*}
and we deduce that
\[
B_\theta = \clsp\{V_p U_n V^*_q : n \in \ZZ^d\text{ and } p,q \in \NN^k\}.
\]
We call $B_\theta$ a \emph{Toeplitz noncommutative torus}.

Now we move on to noncommutative solenoids. First we need some more conventions. We write
$\ESS^d$ for the compact quotient space $\RR^d/\ZZ^d$, and view functions $f\in C(\ESS^d)$
as $\ZZ^d$-periodic continuous functions $f:\RR^d\to \CC$. We write $M(\ESS^d)$ for the
set of positive measures on $\ESS^d$, and view measures $\mu\in M(\ESS^d)$ as positive
functionals $f\mapsto \int_0^1 f\,d\mu$ on $C(\ESS^d)$. Then $\|\mu\|:=\mu(\ESS^d)$ is the
norm of the corresponding functional, and $P(\ESS^d):=\{\mu\in M(\ESS^d):\|\mu\|=1\}$ is
the set of probability measures.

We consider three sequences of matrices $\{\theta_m\}\subset M_{k,d}([0,\infty))$,
$\{D_m\}\subset M_k(\NN)$, and $\{E_m\}\subset M_d(\NN)$ such that: each $D_m$ is
diagonal with entries larger than $1$; each $E_m$ has $\det E_m>1$; and
\begin{equation}\label{relatetheta}
D_m\theta_{m+1}E_m=\theta_m\quad\text{for $m\geq 1$.}
\end{equation}
We choose a sequence $\{r^m\}=\{(r^m_j)\}$ of vectors in $(0,\infty)^k$ satisfying
\begin{equation}\label{relater}
r^{m+1}=D_m^{-1}r^m \quad\text{for $m\geq 1$}
\end{equation}
Notice that both sequences are determined by the first terms $\theta_1\in
M_{k,d}([0,\infty))$ and $r^1\in [0,\infty)^k$.

\begin{example}
We fix $N\geq 2$, and set $d=k=1$, $D_m=E_m=N$ for $m\geq 2$, $\theta_1\in (0,\infty)$
and $\theta_m=N^{-2(m-1)}\theta_1$. Taking the equivalence classes of the $\theta_m$ in
$\ESS = \RR/\ZZ$ yields an example of the set-up of \cite{BHS} except that we are
insisting that $\theta_m = N^2 \theta_{m+1}$ as real numbers, not just as elements of
$\ESS$. This has the consequence that $\theta_m\to 0$ as $m \to \infty$, which need not
happen in the situation of \cite{BHS}; but see Remark~\ref{rmk:theta woes} below.
\end{example}

\begin{remark}\label{rmk:theta woes}
Our hypothesis that $D_m\theta_{m+1}E_m=\theta_m$ exactly, and not just modulo $\ZZ^d$,
seems to be crucial in our arguments. Specifically, to assemble the sequences of KMS
states that we will construct on the approximating subalgebras $B_m$ into a KMS state on
$B_\infty$, we will need to show that the associated probability measures $\nu_m$ (see
Proposition~\ref{prp:KMS char}(\ref{it:subinvariance})) intertwine through the maps
induced by the $E_m^T$. We do this in Lemma~\ref{normalising}, and we indicate there the
step in the first displayed calculation where it is critical that
$D_m\theta_{m+1}E_m=\theta_m$ exactly. This prompted us to review carefully the arguments
of \cite{BHS} and we believe that those arguments also require that $N^2 \theta_{m+1} =
\theta_m$ exactly. Specifically, the calculation at the end of the proof of
\cite[Theorem~6.9]{BHS} implicitly treats $\theta_j$ as an element of $\RR$ (there are
many solutions to $N^k \gamma = \theta_j$ in $\ESS$). Similarly the formulas in
\cite[Section~8]{BHS} that involve setting $r_j = \beta/(N^j\theta_j)$ only make sense if
$\theta_j$ is an element of $\RR$. In particular, in the final displayed calculation in
the proof of \cite[Lemma~8.1]{BHS}, it is critical that $N^2\theta_{j+1} = \theta_j$
exactly.
\end{remark}

For each $m$ there is a Toeplitz noncommutative torus $B_m:=B_{\theta_m}$ with generators
$U_{m,n}$ and $V_{m,p}$ such that: $U:n\mapsto U_{m,n}$ is a unitary representation of
$\ZZ^d$, $V:p\mapsto V_{m,p}$ is a Nica-covariant isometric representation of $\NN^k$,
and the pair $U$, $V$ satisfy the  commutation relation~\eqref{relation:V-U} for the
matrix $\theta_m$.

Next we use the matrices $D_m$ and $E_m$ to build homomorphisms from $B_m$ to $B_{m+1}$.

\begin{prop}
Suppose that $m$ is a positive integer. Then there is a homomorphism $\pi_m:B_m\to
B_{m+1}$ such that $\pi_m(U_{m,n})=U_{m+1,E_mn}$ and $\pi_m(V_{m,p})=V_{m+1,D_mp}$.
\end{prop}

\begin{proof}
We define $U:\ZZ^d \to B_{m+1}$ by $U_n=U_{m+1,E_mn}$  and $V:\NN^k\to B_{m+1}$ by
$V_p=V_{m+1,D_mp}$. Then since $D_m$ and $E_m$ have entries in $\NN$, $U$ is a unitary
representation of $\ZZ^d$ and $V$ is an isometric representation of $\NN^k$.

We claim that $V$ is Nica-covariant. To see this, we take $p,q\in \NN^k$. Then Nica
covariance of $p\mapsto V_{m+1,p}$ implies that
\begin{align}\label{NicacovVm+1}
V_pV_p^*V_{q}V_{q}^*&=V_{m+1,D_mp}V_{m+1,D_mp}^*V_{m+1,D_mq}V_{m+1,D_mq}^*\\
&=V_{m+1,(D_mp)\vee(D_mq)}V_{m+1,(D_mp)\vee(D_mq)}^*.\notag
\end{align}
Now recall that $D_m$ is diagonal\footnote{This is crucial here. For example, consider
\[
D=\begin{pmatrix}1&1\\0&1\end{pmatrix}.
\]
Then $De_1=e_1$, $De_2=e_1+e_2$, $e_1\vee e_2=e_1+e_2$, and $D(e_1\vee e_2) =2e_1+e_2$ is
not the same as $(De_1)\vee(De_2)=e_1+e_2$.}, with diagonal entries $d_{m,j}$, say. Then
for $1\leq j\leq k$ we have
\begin{align*}
\big((D_mp)\vee(D_mq)\big)_j&=\max\{(D_mp)_j,(D_mq)_j\}=\max\{d_{m,j}p_j, d_{m,j}q_j\}\\
&=d_{m,j}\max\{p_j,q_j\}=d_{m,j}(p\vee q)_j\\
&= (D_m(p\vee q))_j.
\end{align*}
Thus
\[
V_{m+1,(D_mp)\vee(D_mq)}=V_{m+1, D_m(p\vee q)}=V_{p\vee q},
\]
and~\eqref{NicacovVm+1} says that $V$ is Nica covariant.

We next claim that $U$ and $V$ satisfy the commutation relation~\eqref{relation:V-U}. We
take $n\in \ZZ^d$ and $p\in \NN^k$, and compute using the commutation relation in
$B_{m+1}$:
\begin{align*}
U_nV_p&=U_{m+1,E_mn}V_{m+1,D_mp}\\
&=e^{2\pi i(D_mp)^T\theta_{m_1}E_mn}V_{m+1,D_mp}U_{m+1,E_mn}\\
&=e^{2\pi ip^T(D_m\theta_{m_1}E_m)n}V_{m+1,D_mp}U_{m+1,E_mn}\\
&=e^{2\pi ip^T\theta_m n}V_pU_n\quad \text{using~\eqref{relatetheta}.}
\end{align*}
Now the universal property of $B_m$ gives the desired homomorphism $\pi_m$.
\end{proof}

\begin{remark}
Although we don't think we use this anywhere, the homomorphisms $\pi_m$ are in fact
injective. One way to see this is to use the Nica covariance of $n\mapsto V_{m,n}$ to get
a homomorphism $\pi_{V_m}:\TT(\NN^k)\to B_{\theta^m}$, and interpret~\eqref{relation:V-U}
as saying that $(\pi_{V_m},U_m)$ is a covariant representation of a dynamical system
$(\TT(\NN^k),\ZZ^d,\gamma^m)$ in the algebra $B_{\theta^m}$. Then $B_{\theta^m}$ has the
universal property which characterises the crossed product
$\TT(\NN^k)\rtimes_{\gamma^m}\ZZ^d$, and we can deduce from the equivariant uniqueness
theorem for the crossed product (for example, \cite[Corollary~4.3]{R}) that the
representation
\[
\pi_{D_m,E_m}:=\pi_{V_{m+1}\circ D_m}\rtimes (U_{m+1}\circ E_m)
\]
of $\TT^k(\NN^k)\rtimes_{\gamma^m} \ZZ^d$ in $\TT^k(\NN^k)\rtimes_{\gamma^{m+1}} \ZZ^d$
is faithful.
\end{remark}

We now define our \emph{higher-rank Toeplitz noncommutative solenoid} to be the direct
limit
\begin{equation}\label{defBinfty}
B_\infty:=\varinjlim_{m\in \NN}(B_m,\pi_m).
\end{equation}
We write $\pi_{m,\infty}$ for the canonical homomorphism of $B_m$ into $B_\infty$. To
ease notation we also write $U_{m,n}$ for the image $\pi_{m,\infty}(U_{m,n})$ in
$B_\infty$.

Now we use the vectors $r^m\in (0,\infty)^k$ from our set-up to define the dynamics we
propose to study.

\begin{prop}
There is a dynamics $\alpha:\RR\to \Aut B_\infty$ such that
\begin{equation}\label{defalpha}
\alpha_t\big(V_{m,p}U_{m,n}V_{m,q}^*\big)=e^{it(p-q)^Tr^m}V_{m,p}U_{m,n}V_{m,q}^*.
\end{equation}
\end{prop}

\begin{proof}
Since $U_m$ and $V_m':p\mapsto e^{itp^Tr^m}V_{m,p}$ satisfy the same relations in $B_m$
as $U_m$ and $V_m$, there is  a dynamics $\alpha^{r^m}:\RR\to \Aut B_m$ such that
\[
\alpha^{r^m}(V_{m,p}U_{m,n}V_{m,q}^*)=e^{it(p-q)^Tr^m}V_{m,p}U_{m,n}V_{m,q}^*.
\]

We claim that $\pi_m\circ \alpha^{r_m}_t=\alpha^{r^{m+1}}_t\circ\pi_m$. To see this, we
compute on generators. First, for  $n\in \ZZ^d$ we have
\begin{align*}
\alpha^{r^{m+1}}_t(\pi_m(U_{m,n}))
    &= \alpha^{r^{m+1}}_t(U_{m+1, E_m n})
    = U_{m+1, E_m n}\\
    &= \pi_m(U_{m,n})
    = \pi_m(\alpha^{r^m}_t(U_{m,n})).
\end{align*}
Second, for $p\in \NN^k$, and using the relation~\eqref{relater} at the crucial step to
pass from $r^{m+1}$ to $r^m$, we have
\begin{align*}
\alpha^{r^{m+1}}_t(\pi_m(V_{m,p}))
    &= \alpha^{r^{m+1}}_t(V_{m+1, D_m p})
    = e^{it (D_m p)^T r^{m+1}} V_{m+1, D_m p} \\
    &= e^{itp^T D_m r^{m+1} } \pi_m(V_{m, p})
    = e^{it p^T r^m } \pi_m(V_{m, p})\\
    &= \pi_m(\alpha^{r^m}_t(V_{m,p})).
\end{align*}

Now the universal property of the direct limit implies that for each $t\in \RR$, there is
an automorphism $\alpha_t$ of $B_\infty$ such that $\alpha_t\circ
\pi_{m,\infty}=\pi_{m,\infty}\circ \alpha^{r^m}_t$. The formula~\eqref{defalpha} (which
implicitly involves the homomorphisms $\pi_{m,\infty}$) implies that $t\mapsto \alpha_t$
is a strongly continuous action of $\RR$ on $B_\infty$.
\end{proof}

Our goal is to describe the KMS states of the dynamical system $(B_\infty,\alpha)$. But
first we pause to establish some conventions about probability measures on inverse
limits.

\begin{remark}
All measures in this paper are positive Borel measures. We view probability measures on a
compact space $X$ as states on the $C^*$-algebra $C(X)$ of continuous functions.  We
write $P(X)$ for the set of probability measures on $X$.

When $\big\{h_m:X_{m+1}\to X_m:m\in \NN\big\}$ is an inverse system of compact spaces
with each $h_m$ surjective, the inverse limit $\varprojlim (X_m,h_m)$ is also a compact
space. We write $h_{m,\infty}$ for the canonical map of $X_\infty:=\varprojlim (X_m,h_m)$
onto $X_m$, so that we have $h_{m,\infty}=h_m\circ h_{m+1,\infty}$ for all $m\in \NN$.
The maps $h_{m,\infty}$ induce maps $h_{m,\infty*}$ on measures: if $\mu$ is a
probability measure on $X_\infty$, then $\mu_m:=h_{m,\infty*}(\mu)$ is the measure on
$X_m$ such that
\[
\int_{X_m}f\,d\mu_m=\int_{X_\infty} (f\circ h_{m,\infty})\,d\mu\quad\text{for $f\in C(X_m)$.}
\]
Conversely, because each $h_m$ is surjective, for any sequence of probability measures
$\{\mu_m\in P(X_m):m\in \NN\}$ such that $\mu_m=h_{m*}(\mu_{m+1})$ for all $m$ there is a
probability  measure $\mu\in P(X_\infty)$ such that $\mu_m=h_{m,\infty*}(\mu)$ for all
$m$ (see \cite[Lemma~6.1]{BLPRR}, for example). Thus the simplices $P(\varprojlim X_m)$
and $\varprojlim P(X_m)$ are canonically isomorphic.
\end{remark}

To state our main result, we need to observe that, because the entries in the $E_m$ are
integers, multiplication by $E_m^T$ on $\RR^d$ maps $\ZZ^d$ into $\ZZ^d$ and hence
induces a homomorphism $E_m^T$ of $\ESS^d=\RR^d/\ZZ^d$ onto itself. We show that the KMS
states are parametrised by the probability measures on the inverse limit $\varprojlim
(\ESS^d, E^T_m)$, which is an ordinary solenoid. We write $E^T_{m,\infty}$ for the
projection of $\varprojlim (\ESS^d, E^T_m)$ on the $m$th copy of $\ESS^d$, so that we have
\[
E^T_{m,\infty}=E^T_m\circ E^T_{m+1, \infty} \quad\text{for $m\in \NN$.}
\]

The main theorem of this paper is the following; we prove it at the end of the paper.

\begin{thm}\label{newbigthm}
Suppose that $\mu\in P\big(\varprojlim(\ESS^d,E_m^T)\big)$ and $\beta>0$. Let $\{\mu_m\}$
be the corresponding sequence of probability measures on $\ESS^d$. For $m\in \NN$ and
$n\in \NN^d$, we define the \emph{moment} $M_{m,n}(\mu)$ to be the number
\[
M_{m,n}(\mu)=\int_{\ESS^d} e^{2\pi ix^Tn}\,d\mu_m(x)=\int_{\varprojlim(\ESS^d,E_m^T)} e^{2\pi iE^T_{m,\infty}(x)^Tn}\,d\mu(x).
\]
Then there is a KMS$_\beta$ state $\psi_\mu$ on $(B_{\infty},\alpha)$ such that
\begin{equation}\label{explicitform}
\psi_\mu(V_{m,p} U_{m,n} V_{m,q}^*)=\delta_{p,q}e^{-\beta p^Tr^m}\prod_{j=1}^k\frac{\beta r^m_j}{\beta r^m_j-2\pi i(\theta_m^Tn)_j}M_{m,n}(\mu).
\end{equation}
The map $\mu\mapsto \psi_\mu$ is an affine homeomorphism of $P\big(\varprojlim(\ESS^d,
E_m^T)\big)$ onto the simplex KMS$_\beta (B_\infty, \alpha)$ of KMS$_\beta$ states.
\end{thm}

\begin{remark}
As a reality check, we take $p=q=0$ and $n=0$. Then $V_{m,p}U_{m,n}V_{m,q}^*$ is the
identity $1_{B_m}=1_{B_\infty}$, and our formula collapses to $\psi_\mu(1)=1$.
\end{remark}

\begin{remark}
It is interesting to set $d=k=1$ and compare the formula~\eqref{explicitform} with the
formula (6.4) in Theorem~6.9 of \cite{BHS}, which on the face of it looks different. The
point is that the integral on the right-hand side of \cite[(6.4)]{BHS} is with respect to
the subinvariant measure associated to the probability measure $\mu$, which in our
notation would be $\nu_{\mu_m}$. There is no specific description for this measure in
\cite{BHS}: they get an isomorphism of the simplex $P(\varprojlim \ESS)$ onto the simplex
of subinvariant measures by specifying it on the extreme points (see
\cite[Lemma~8.2]{BHS}). We reconcile the two approaches in Remark~\ref{reconciliation}.
\end{remark}

\section{Equilibrium states on a Toeplitz noncommutative torus}\label{sec:NCT-new}

In this section, we fix $\theta\in M_{k,d}([0,\infty))$, and investigate the KMS states
on the Toeplitz noncommutative torus $B_\theta$.

For $n\in \ZZ^d$, we write $g_n$ for the character on $\ESS^d$ given by $g_n(x)=e^{2\pi
ix^Tn}$, and $\iota: C(\ESS^d)\to C^*(\ZZ^d)\subset B_\theta$ for the isomorphism  such
that $\iota(g_n)=U_n$. Then we have
\[
B_\theta=\clsp\big\{V_p\iota(f) V_q^*:f\in C(\ESS^d), p,q\in\NN^k\big\}.
\]
For $y\in\RR^d$ we define $R_y: \ESS^d\to \ESS^d$ by $R_y(x)=x+y$. Later, we will also
write $R_y^*$ for the automorphism of $C(\ESS^d)$ given by $R_y^*f=f\circ R_y$, and
$R_{y*}$ for the dual map on measures defined by
\[
\int_{\ESS^d} f\,dR_{y*}(\mu)=\int_{\ESS^d} R_y^*(f)\,d\mu=\int_{\ESS^d} f\circ R_y\,d\mu.
\]
The assignment $y\mapsto R_y^*$ is a strongly continuous action $R$ of $\RR^d$ on
$C(\ESS^d)$, and each $R_{y*}$ is norm-preserving.

\begin{lemma}\label{lem:action on fns}
For $f \in C(\ESS^d)$ and $p \in \NN^k$ we have
\begin{equation}\label{def-R}
V_p \iota(f) = \iota\big(f\circ R_{-\theta^T p}\big) V_p \quad\text{ and }\quad V^*_p \iota(f) = \iota\big(f\circ R_{\theta^T p}\big) V^*_p.
\end{equation}
\end{lemma}

\begin{proof}
Since $C(\ESS^d)=\clsp\{g_n:x\mapsto e^{2\pi ix^Tn}: n\in \ZZ^d\}$, it suffices to
check~\eqref{def-R} for $f=g_n$. Let $n \in \ZZ^d$. Then~\eqref{relation:V-U} gives
\[
V_p\,\iota(g_n)= V_p U_n= e^{-2\pi i p^T\theta n} U_n V_p=e^{-2 \pi i p^T \theta n} \iota(g_n) V_p.
\]
Since
\[
e^{-2 \pi i p^T \theta n}g_n(x)=e^{-2 \pi i p^T \theta n}e^{2\pi ix^Tn}=g_n(x-\theta^Tp)=(g_n\circ R_{-\theta^Tp})(x),
\]
the first equality follows. The second follows from a similar computation
using~\eqref{relation:V*-U}.
\end{proof}

\begin{remark}
The minus sign in the first identity in~\eqref{def-R} is crucial. As a reality check,
notice that the signs in the two formulas have to be different, because $V_p^*V_p=1$
means the $\pm\theta^Tp$ have to cancel. As a corollary, note that $V_pV_p^*$, which is a
proper projection, commutes with the $\iota(f)$. (To see that $V_pV_p^*\not=1$, we can
use the specific representation of $B_\theta$ constructed in the proof of
Proposition~\ref{prp:KMS char}(b).)
\end{remark}

 We now fix $r\in(0,\infty)^k$. The universal property of $B_\theta$ gives a dynamics $\alpha^r:\RR\to\Aut B_\theta$ such that
\begin{align}\label{action-tori}
\alpha^r_t(U_n) = U_n\quad\text{and}\quad \alpha^r_t(V_p) = e^{it p^Tr} V_p \quad\text{for $n \in \ZZ^d$, $p \in \NN^k$, $t \in \RR$.}
\end{align}
Then $\alpha_t^r(V_p U_n V^*_q) = e^{it(p-q)^Tr} V_p U_nV^*_q$, and hence \[\{V_p U_n
V^*_q : n \in \ZZ^d, p,q \in \NN^k\}\] is a set of $\alpha^r$-analytic elements spanning
an $\alpha^r$-invariant dense subset of $B_\theta$.

To describe the KMS$_\beta$ states of $(B_\theta,\alpha^r)$, it was tempting to apply
\cite[Theorem~6.1]{AaHR2} to the Toeplitz algebra of the commuting homeomorphisms
$h_j:x\mapsto x+\theta_j$ associated to the rows $\theta_j$ of $\theta$. That result is
in several ways more general than we need, but has an unfortunate hypothesis of rational
independence on the set $\{r_j\}$ which we prefer to avoid.

\begin{prop}\label{prop:KMS formula}
Suppose that $\beta>0$ and $\phi$ is a KMS$_\beta$ state of $(B_\theta,\alpha^r)$. Then
$\phi$ is a KMS$_\beta$ state of $(B_\theta,\alpha^r)$ if and only if
\begin{equation}\label{eq:KMS char}
\phi(V_p U_n V^*_q) = \delta_{p,q} e^{- \beta p^Tr} \phi(U_n)\quad\text{ for }n \in \ZZ^d\text{ and }p,q \in \NN^k.
\end{equation}
\end{prop}

To prove Proposition~\ref{prop:KMS formula} we need two lemmas. The arguments are based
on the proofs of Lemmas~5.2 and~5.3 in \cite{aHLRS}.

\begin{lemma}\label{lem:CS}
Suppose that $\beta>0$ and $\phi$ is a KMS$_\beta$ state of $(B_\theta,\alpha^r)$.  If
$p,q \in \NN^k$ satisfy $p^Tr = q^Tr$, then
\begin{enumerate}
\item\label{it:q->p}  $\phi(V_p U_n V^*_p) = \phi(V_qU_n V^*_q)$ for $n\in \ZZ^d$;
    and
\item\label{it:cs estimate}$|\phi(V_p\,\iota(f) V^*_q)| \le \phi(V_p \,\iota(f)
    V^*_p)$ for positive $f\in C(\ESS^d)$.
\end{enumerate}
\end{lemma}
\begin{proof}
For (\ref{it:q->p}),  since $V_q$ is an isometry,  we have
\[
\phi(V_p U_n V^*_p)= \phi\big(V_p U_n(V_q^* V_q)V^*_p\big)= \phi\big((V_p U_n V_q^*)( V_q V^*_p)\big),
\]
and since $p^Tr = q^Tr$ the KMS condition gives
\[
\phi(V_p U_n V^*_p)= e^{- \beta (p-q)^Tr }\phi\big(V_q V^*_p(V_p U_n V^*_q)\big)= \phi(V_q U_n V^*_q).
\]

For (\ref{it:cs estimate}), we take a positive function $f$ in $C(\ESS^d)$.  By linearity
and continuity,  part~(\ref{it:q->p}) implies that $\phi(V_q\,\iota( f )V^*_q) = \phi(V_p
\,\iota( f ) V^*_p)$.  Using the Cauchy--Schwarz inequality at the second step, we
calculate:
\begin{align*}
|\phi(V_p\,\iota( f )V^*_q)|^2
& = \big|\phi\big((V_p \,\iota(\sqrt{f})) (V_q\,\iota(\sqrt{f}))^*\big)\big|^2
\\
&\le \phi(V_p \,\iota( f ) V^*_p) \phi(V_q \,\iota( f ) V^*_q)
\\
&=\phi(V_p \,\iota( f ) V^*_p)^2.
\end{align*}
Since both sides are the squares of non-negative numbers, we can take square roots, and
we retrieve (\ref{it:cs estimate}).
\end{proof}

\begin{lemma}\label{lem:flipflop}
Suppose that $\beta>0$ and $\phi$ is a KMS$_\beta$ state of $(B_\theta,\alpha^r)$.
Suppose that $p,q \in \NN^k$ satisfy $p^Tr=q^Tr$ and that $f \in C(\ESS^d)$. Write $P:=(p
\vee q)-p$. Then
\begin{align}\label{eq:flipflop}
\phi\big(V_p \,\iota(f) V^*_q\big) = \phi\big(V_{p+lP}  \,\iota( f \circ R_{l\theta^TP}) V_{q + lP}^*\big)\quad\text{for all  $l \in \NN$}.
\end{align}
If $p \not= q$, then $\phi(V_p \,\iota(f) V^*_q) = 0$.
\end{lemma}

\begin{proof}
We prove~\eqref{eq:flipflop} by induction on $l$. The base case $l = 0$ is trivial. Now
suppose that~\eqref{eq:flipflop} holds for $l \ge 0$. The inductive hypothesis gives
\begin{align*}
\phi(V_p\,\iota(f) V^*_q)
&= \phi\big(V_{p+lP} \,\iota(f \circ R_{l\theta^TP}) V^*_{q+lP}\big)\\
&=\phi\big(V_{p + lP} \,\iota(f \circ R_{l\theta^TP}) V^*_{q + lP}V_{q + lP}V^*_{q + lP}\big).
\end{align*}
Since the dynamics $\alpha^r$ fixes the element $V_{q + lP}V^*_{q + lP}$, the KMS
condition implies that
\[
\phi(V_p\iota(f)V_q^*)=\phi\big(V_{q + lP} V^*_{q + lP} V_{p+lP} \,\iota(f \circ R_{l\theta^TP}) V^*_{q + lP}\big),
\]
and Nica covariance gives
\begin{align*}
\phi(V_p&\,\iota(f) V^*_q)\\
&= \phi\big(V_{q + lP} V_{((q + lP) \vee(p + lP)) - (q+lP)} V^*_{((q + lP) \vee(p + lP)) - (p + lP)} \,\iota(f \circ R_{l\theta^TP}) V^*_{q + lP}\big).
\end{align*}
For $c\in \NN^k$ we have $(p+c)\vee (q+c)=(p\vee q)+c$. Thus
	\begin{align*} \phi(V_p\,\iota(f) V^*_q)	
    &= \phi\big(V_{q + lP} V_{(p \vee q) - q} V^*_{(p \vee q) - p} \,\iota(f \circ R_{l\theta^TP}) V^*_{q + lP}\big)\\
    &= \phi\big(V_{(p \vee q) + lP} V^*_P \,\iota(f \circ R_{l\theta^TP}) V^*_{q + lP}\big)\\
&= \phi\big(V_{(p \vee q) + lP} \,\iota(f \circ R_{l\theta^TP}\circ R_{\theta^T P}) V^*_{q + (l+1)P}\big)\quad\text{by  Lemma~\ref{lem:action on fns}}\\
    &= \phi\big(V_{p + (l+1)P} \,\iota(f \circ R_{(l+1)\theta^TP}) V^*_{q + (l+1)P}\big)
\end{align*}
because $(p \vee q) + lP=p+(l+1)P$. This completes the inductive step, and hence the
proof of~\eqref{eq:flipflop}.

Now suppose that $p\neq q$. Then at least one of $P$ and $(p\vee q) - q$ is nonzero. We
argue the case where $P \not= 0$, and the other case follows by taking adjoints. For
$l\in\NN$ we have
\begin{align*}
|\phi(V_p\,\iota( f) V_q^*)|
&= \big|\phi\big(V_{p + lP} \,\iota(f \circ R_{l\theta^TP}) V^*_{q + lP}\big)\big|\\
&\le \phi\big(V_{p+lP} \,\iota(f \circ R_{l\theta^TP}) V^*_{p + lP}\big) \quad\text{by Lemma~\ref{lem:CS}(\ref{it:cs estimate})}\\
&= e^{ -\beta(p+lP)^Tr} \phi\big(V^*_{p + lP}V_{p + lP}\iota(f \circ R_{l\theta^TP})\big)\\
&= e^{ -\beta (p+lP)^Tr} \phi\big(\iota(f \circ R_{l\theta^TP})\big)\\
&\leq  e^{ -\beta (p+lP)^Tr}\|f\|_{\infty}.
\end{align*}
Since $P > 0$ and $r \in (0,\infty)^k$, we have $(p + lP)^Tr \to \infty$ as $l \to
\infty$, and hence $e^{-\beta(p+lP)^Tr} \|f\|_{\infty} \to 0$ as $l \to \infty$. Thus
$\phi(V_p\,\iota( f) V_q^*)=0$.
\end{proof}

\begin{proof}[Proof of Proposition~\ref{prop:KMS formula}]
First suppose that $\phi$ is a KMS$_\beta$ state for $(B_\theta,\alpha^r)$. For $n \in
\ZZ^d$ and $p,q \in \NN^k$, two applications of the KMS condition give
\begin{equation}\label{eq:2flips}
\phi(V_p U_n V^*_q)
    = e^{-\beta p^T r} \phi(U_n V^*_q V_p)
    = e^{-\beta(p-q)^T r } \phi(V_p U_n V^*_q).
\end{equation}
It follows immediately that if $(p-q)^Tr \not= 0$, then $\phi(V_p U_n V^*_q) = 0$. If
$(p-q)^Tr = 0$ but $p \not= q$, then Lemma~\ref{lem:flipflop} gives $\phi(V_p U_n V^*_q)
= 0$. This combined with the first equality in~\eqref{eq:2flips} gives
\[\phi(V_p U_n V^*_q) = \delta_{p,q} e^{-\beta p^T r}
\phi(U_n V^*_q V_p)=\delta_{p,q} e^{-\beta p^T r}
\phi(U_n)
\]
because $V_p$ is an isometry. This is the desired formula~\eqref{eq:KMS char}.

Now suppose that $\phi$ is a state satisfying~\eqref{eq:KMS char}. Since the $V_p
U_nV^*_q$ are analytic elements spanning a dense $\alpha^r$-invariant subspace of
$B_\theta$, it suffices to fix $p,q,b,c \in \NN^k$ and $n,n' \in \ZZ^d$, and show that
\begin{equation}\label{eq:KMS to show}
\phi(V_p U_n V^*_q V_b U_{n'} V^*_c) = e^{-\beta(p-q)^T r} \phi(V_b U_{n'} V^*_c V_p U_n V^*_q).
\end{equation}
Let $P := (q \vee b) - b$ and $Q := (q \vee b) - q$. Then $P,Q \in \NN^k$ are the unique
elements such that $P \wedge Q=0$ and $ P+b= Q +q$, and Nica covariance says that
$V_q^*V_b= V_QV_P^*$. Now we calculate,  using first the identities
\eqref{relation:V-U}~and~\eqref{relation:V*-U}, and then (at the last step) the
assumption~\eqref{eq:KMS char}:
\begin{align}
\phi(V_p U_n V^*_q V_b U_{n'} V^*_c)
&= \phi(V_p U_n V_Q V^*_P U_{n'} V^*_c) \notag\\
&= e^{2\pi iQ^T\theta n} \phi(V_p(V_{Q}U_{n}) V^*_P U_{n'} V^*_c)\notag\\
&= e^{2\pi i (Q^T\theta n + P^T \theta n')} \phi(V_{Q+p} U_{n+n'} V^*_{P+c})\notag\\
&= \delta_{Q+p, P+c} e^{-\beta(Q+p)^Tr } e^{2\pi i  (Q^T\theta n + P^T \theta n')} \phi(U_{n+n'}).\label{eq:LHS}
\end{align}

Similarly, let $M := (c \vee p) - p$ and  $N := (c \vee p) - c$. Then $M,N \in \NN^k$ are
the unique elements such that $M \wedge N=0$ and $M+p=N+c$, and the right-hand side
of~\eqref{eq:KMS to show} is
\begin{align}
    e^{-\beta(p-q)^T r} &\phi(V_b U_{n'} V_N V^*_M U_n V^*_q)\nonumber\\
    &= e^{-\beta(p-q)^T r} e^{2\pi i  (N^T\theta {n'} + M^T\theta n)} \phi(V_{b+N} U_{n+n'} V^*_{q+M})\nonumber\\
    &= \delta_{N+b, M+q} e^{- \beta(p-q+b+N)^Tr} e^{2\pi i(N^T\theta {n'} + M^T\theta n)} \phi(U_{n+n'}).\label{eq:RHS}
\end{align}
To see  that~\eqref{eq:LHS} is equal to~\eqref{eq:RHS}, we first show that the two
Kronecker deltas have the same value. For this, observe that by definition of $M,N,P,Q$,
we have
\[
( P+b) + (N+c) = (Q+q) + (M+p),
\]
and consequently $(N+b) -( M+q) =  ( Q+p)-( P+c)$. Thus $\delta_{Q+p, P+c} = 1$ if and
only if $\delta_{N+b, M+q} = 1$. So it now suffices to prove that~\eqref{eq:LHS}
equals~\eqref{eq:RHS} when $Q+p = P+c$ and $ N+b =  M+q$.

We first claim that $M=Q$ and $N=P$. By assumption, we have $  M+q=N+b$, and we have
$P+b=Q+q$ by definition of $P,Q$. Subtracting these equations, we obtain $M-Q=N-P$, and
rearranging gives $M-N=Q-P$. Since $P \wedge Q=0$ and $M \wedge N=0$, we deduce that
$Q=(Q-P)\vee 0=(M-N)\vee 0 =M$, and then $P=N$ too, as claimed.

We now have
\[
e^{2\pi i (Q^T\theta n + P^T \theta n')} = e^{2\pi i (M^T\theta n+N^T\theta n')},
\]
and so it remains to check that
\[
e^{- \beta (p-q+b+N)^T r} = e^{-  \beta(p+Q)^T r}.
\]
For this, we apply $N=P$, from above, at the second equality and  $b + P = q + Q$, by
definition of $Q,P$, at the third  to get
\begin{align*}
(p - q) + (b + N)&= p + (b + N - q)= p + (b + P - q)\\
    &= p + (q + Q - q)= p + Q,
\end{align*}
which gives the result. Thus $\phi$ is a KMS$_\beta$ state.
\end{proof}

\begin{lemma}\label{power series of Rs}
Write $\theta_j$ for the $j$th row of $\theta$. Then the series
\begin{equation}\label{opvaluedsum}
\sum_{p\in\NN^k} e^{-\beta p^Tr}R_{\theta^Tp*}
\end{equation}
converges in the operator norm of $B(C(\ESS^d))$ to an inverse for $\prod^k_{j=1} (\id -
e^{- \beta r_j}R_{\theta^T_j*})$.
\end{lemma}

\begin{proof}
We first need to understand the sum~\eqref{opvaluedsum}, which we want to calculate as an
iterated sum. So we interpret~\eqref{opvaluedsum} as a $B(C(\ESS^d))$-valued integral over
$\NN^k$ with respect to counting measure $\sigma$ (for which all functions on $\NN^k$ are
measurable). Since each $R_{\theta^Tp}$ is norm-preserving, we have
\[
\big\|e^{-\beta p^Tr}R_{\theta^Tp*}\big\|=e^{-\beta p^Tr}=\prod_{j=1}^ke^{-\beta p_jr_j}.
\]
By Tonelli's theorem, we have
\begin{align*}
\sum_{p\in\NN^k}\big\|e^{-\beta p^Tr}R_{\theta^T_jp*}\big\|
&=\sum_{p_k=0}^\infty\cdots\sum_{p_1=0}^\infty\Big(\prod_{j=1}^ke^{-\beta p_jr_j}\Big)\\
&=\sum_{p_k=0}^\infty\cdots\sum_{p_2=0}^\infty\Big(\prod_{j=2}^ke^{-\beta p_jr_j}\Big)\Big(\sum_{p_1=0}^\infty e^{-\beta p_1r_1}\Big)\\
&=\sum_{p_k=0}^\infty\cdots\sum_{p_2=0}^\infty\Big(\prod_{j=2}^ke^{-\beta p_jr_j}\Big)(1-e^{-\beta r_1})^{-1}.
\end{align*}
Repeating this $k-1$ more times gives
\[
\sum_{p\in\NN^k}\big\|e^{-\beta p^Tr}R_{\theta^T_jp*}\big\| = \prod_{j=1}^k(1-e^{-\beta r_j})^{-1}.
\]
Thus the function $p\mapsto e^{-\beta p^Tr}R_{\theta^T_jp*}$ is integrable with respect
to $\sigma$, and Fubini's theorem for functions with values in a Banach space (for
example, \cite[Theorem~II.16.3]{FD}) implies that
\begin{align*}
\sum_{p\in\NN^k} e^{-\beta p^Tr}R_{\theta^Tp*}&=\sum_{p_k=0}^\infty\cdots\sum_{p_1=0}^\infty\Big(\prod_{j=1}^k e^{-\beta p_jr_j}R_{p_j\theta_j^T*}\Big)\\
&=\prod_{j=1}^k\Big(\sum_{p_j=0}^\infty\big(e^{-\beta r_j}R_{\theta_j^T*}\big)^{p_j}\Big).
\end{align*}
Writing the infinite sum as a limit of partial sums shows that
\begin{equation}\label{inverseforj}
\sum_{p_j=0}^\infty\big(e^{-\beta r_j}R_{\theta_j^T*}\big)^{p_j}\big(\id-
e^{-\beta r_j}R_{\theta_j^T*}\big)=\id.
\end{equation}
To simplify the product
\[
\bigg(\prod_{j=1}^k\Big(\sum_{p_j=0}^\infty\big(e^{-\beta r_j}R_{\theta^T*}\big)^{p_j}\Big)\bigg)\Big(\prod_{j=1}^k\big(\id-
e^{-\beta r_j}R_{\theta_j^T*}\big)\Big),
\]
we write the left-hand product from $j=k$ to $j=1$, and the right-hand one from $j=1$ to
$j=k$. Now $k$ applications of~\eqref{inverseforj} show that the product telescopes to
the identity $\id$ of $B(C(\ESS^d))$.
\end{proof}

The next proposition is an analogue of \cite[Theorem~6.1]{aHLRS} and
\cite[Theorem~6.1]{AaHR2}.

\begin{prop}\label{prp:KMS char}
Fix $\beta\in(0,\infty)$.
\begin{enumerate}
\item\label{it:subinvariance} Suppose that $\phi$ is a KMS$_\beta$ state for
    $(B_\theta,\alpha^r)$, and let $\nu \in P(\ESS^d)$ be the measure such that
\[
\phi(\iota(f)) =\int_{\ESS^d} f\,d\nu\quad\text{for $f \in C(\ESS^d)$.}
\]
Suppose that $F \subset \NN^k$ is a finite set such that $p\neq q\in F$ implies
$p\wedge q=0$. Then the measure $\nu$ satisfies the \emph{subinvariance relation}
\begin{equation}\label{eq:subinv}
        \prod_{p \in F} \big(\id - e^{-\beta p^Tr} R_{\theta^Tp*}\big)(\nu) \ge 0.
    \end{equation}

\item\label{it:phi-mu} Define $y_\beta:=\sum_{p\in \NN^k}e^{-\beta p^Tr}$, and
    suppose that $\kappa$ is a positive measure on $\ESS^d$ with total mass
    $y_\beta^{-1}$. Write $\theta_j$ for the $j$th row of $\theta$. Then
\[
\nu=\nu_\kappa:=\prod_{j=1}^k\big(\id-e^{-\beta r_j}R_{\theta_j*}\big)^{-1}(\kappa)
\]
is a subinvariant probability measure, and  there is a KMS$_\beta$ state $\phi_\nu$
of $(B_\theta, \alpha^r)$ such that
\begin{align}\label{KMS from mu m}
\phi_\nu\big(V_p\,\iota(f)V^*_q\big)=\delta_{p,q}e^{-\beta p^Tr} \int_{\ESS^d}f\,d\nu\quad\text{for $p,q\in\NN^k$ and $f\in C(\ESS^d)$.}
\end{align}

\item\label{it:isomorphism} The map $\kappa\mapsto \phi_{\nu_\kappa}$ is an affine
    isomorphism  of the simplex
\[
\Sigma_{\beta,r}=\big\{\text{positive measures } \kappa:\|\kappa\|=y_\beta^{-1}\big\}
\]
onto the simplex of KMS$_\beta$ states of $(B_\theta,\alpha^r)$.
\end{enumerate}
\end{prop}

\begin{proof}
(\ref{it:subinvariance}) We take a positive function $f \in C(\ESS^d)_+$ and compute
\begin{align}\label{expandprod}
\int_{\ESS^d} f\,d\Big(\prod_{p \in F}(\id - &e^{-\beta p^T r} R_{\theta^T p*})  (\nu)\Big)
= \int_{\ESS^d} f\circ\Big(\prod_{p \in F}(\id - e^{-\beta p^T r} R_{\theta^T p})\Big)\,d\nu\\
&=\int_{\ESS^d}\sum_{S\subset F} (-1)^{|S|}\Big(\prod_{p\in S}e^{-\beta p^Tr}\Big)\Big(f\circ \prod_{p\in S}R_{\theta^Tp}\Big)\,d\nu.\notag
\end{align}
We write $p_S:=\sum_{p\in S}p$, and observe that $\prod_{p\in S}e^{-\beta p^Tr}=e^{-\beta
p_S^Tr}$ and $\prod_{p\in S}R_{\theta^Tp}=R_{\theta^Tp_S}$. Thus
\begin{align*}
\eqref{expandprod}&=\int_{\ESS^d}\sum_{S\subset F} (-1)^{|S|}e^{-\beta p_S^Tr}\big(f\circ R_{\theta^Tp_S}\big)\,d\nu\\
&=\phi\Big(\sum_{S\subset F} (-1)^{|S|}e^{-\beta p_S^Tr}\iota\big(f\circ R_{\theta^Tp_S}\big)V_{p_S}^*V_{p_S}\Big)\quad\text{since $V_{p_S}^*V_{p_S}=1$}\\
&=\phi\Big(\sum_{S\subset F} (-1)^{|S|}V_{p_S}\iota\big(f\circ R_{\theta^Tp_S}\big)V_{p_S}^*\Big)\quad\text{by the KMS condition}\\
&=\phi\Big(\sum_{S\subset F} (-1)^{|S|}V_{p_S}V_{p_S}^*\iota(f)\Big)\quad\text{by~\eqref{def-R}.}
\end{align*}
Because the set $F$ has the property that $p\wedge q=0$ for $p\not= q\in F$, Nica
covariance gives $V_{p}V_p^*V_qV_{q}^*=V_{p\vee q}V_{p\vee q}^*=V_{p+q}V_{p+q}^*$ for
$p\not=q\in F$.  Thus for each $S\subset F$, we have $V_{p_S}V_{p_S}^*=\prod_{p\in
S}V_pV_p^*$, and
\[
\sum_{S\subset F} (-1)^{|S|}V_{p_S}V_{p_S}^*=\prod_{p\in F}(1-V_pV_p^*).
\]
The latter product is a projection, and it is fixed by the action $\alpha$. Hence another
application of the KMS condition gives
\begin{align*}
\int_{\ESS^d} f\,d\Big(\prod_{p \in F}(\id - &e^{-\beta p^T r} R_{\theta^T p*})  (\nu)\Big)=\phi\Big(\prod_{p\in F}(1-V_pV_p^*)\iota(f)\Big)\\
&=\phi\Big(\Big(\prod_{p\in F}(1-V_pV_p^*)\Big)^2\iota(f)\Big)\\
&=\phi\Big(\prod_{p\in F}(1-V_pV_p^*)\iota(f)\prod_{p\in F}(1-V_pV_p^*)\Big).
\end{align*}
This last term is positive because the argument of $\phi$ is a positive element of
$B_\theta$, and this proves (\ref{it:subinvariance}).

(\ref{it:phi-mu}) We have
\[
\prod_{j=1}\big(\id-e^{-\beta r_j}R_{\theta_j*}\big)(\nu)=\kappa\geq 0,
\]
so $\nu$ is subinvariant.   By Lemma~\ref{power series of Rs} we have
\begin{align}
\int_{\ESS^d}1\, d\nu&=\int_{\ESS^d}1\ d\Big(\sum_{p \in \NN^k} e^{- \beta p^T} R_{\theta^T p*}(\kappa)\Big)\label{normkappa}\\
&=\sum_{p \in \NN^k} e^{- \beta p^Tr} \int_{\ESS^d} 1\circ R_{\theta^Tp}\, d\kappa\notag\\
&=\sum_{p \in \NN^k} e^{- \beta p^T r} \|\kappa\|=y_\beta\|\kappa\|=1,\notag
\end{align}
and hence $\nu$ is a probability measure.

We will build a KMS$_\beta$ state using a representation of $B_\theta$ on $\ell^2(\NN^k)
\otimes L^2(\ESS^d, \kappa)$. Recall that we write $g_n$ for the trigonometric polynomial
$g_n(x)=e^{2\pi ix^Tn}$. Then the formula $W_nf := g_n f$ defines a unitary
representation $W$ of $\ZZ^d$ on $L^2(\ESS^d,\kappa)$. Write $\{\delta_p : p \in \NN^k\}$
for the orthonormal basis of point masses for $\ell^2(\NN^k)$, and let $D_n$ be  the
bounded operator such that $D_n\delta_p := e^{2\pi ip^T\theta n} \delta_p$. Then $D$ is a
unitary representation of $\ZZ^d$ on $\ell^2(\NN^k)$, and hence $D \otimes W$ is a
unitary representation of $\ZZ^d$ on $\ell^2(\NN^k) \otimes L^2(\ESS^d, \kappa)$.

Let $T$ be the usual Toeplitz representation of $\NN^k$ by isometries on $\ell^2(\NN^k)$.
Then we have
\begin{align*}
(T_p \otimes 1)(D_n \otimes W_n)(\delta_q \otimes f)
=e^{2\pi i q^T\theta n}(\delta_{p+q} \otimes W_n f),
\end{align*}
and
\begin{align*}(D_n \otimes W_n)(T_p \otimes 1)(\delta_q \otimes f)
    &= e^{2\pi i (p+q)^T\theta n}(\delta_{p+q} \otimes W_n f)\\
    &=e^{2\pi i p^T\theta n}(T_p \otimes 1)(D_n \otimes W_n).
\end{align*}
Hence  the universal property of $B_\theta$ gives a representation
\[
\pi:B_\theta\to B(\ell^2(\NN^k)\otimes L^2(\ESS^d,\kappa))
\]
such that $\pi(U_n)=(D_n \otimes W_n)$ and $\pi(V_p)=T_p\otimes 1$.

Since $\sum_{p\in \NN^k}e^{-\beta p^Tr}$ is convergent, there is a positive linear
functional $\phi_\nu :B_\theta \to \CC$ such that
\[
\phi_\nu(a):=\sum_{p\in\NN^k}e^{-\beta p^Tr}\big(\pi(a)(\delta_p\otimes 1) \mid \delta_p \otimes 1\big).
\]
Then~\eqref{normkappa} implies that $\phi_\nu(1)=1$, and $\phi_\nu$ is a state. To see
that $\phi_\nu$ is a KMS$_\beta$ state, we take $p,q \in \NN^k$, $n \in \ZZ^d$, and
calculate:
\begin{align}\label{formphinu}
\phi_\nu(V_p \iota(g_n) V_q^*)&=\phi_\nu(V_p U_n V_q^*)\\
    &=\sum_{b\in \NN^k} e^{-\beta b^Tr}  \big((D_n \otimes W_n)(T^*_q\delta_b \otimes 1) \mid T^*_p\delta_b \otimes 1\big)\notag\\
    &= \sum_{b \ge p \vee q} e^{-  \beta b^Tr}  \big(e^{ 2\pi i (b-q)^T\theta n} \delta_{b-q} \otimes g_n \mid \delta_{b-p}\otimes 1\big)\notag\\
    &= \delta_{p,q}\sum_{b \ge p} e^{-\beta b^Tr} e^{2\pi i(b-p)^T\theta n} \big(g_n \mid 1\big)\notag\\
    &= \delta_{p,q}  \Big(\sum_{b \in \NN^k} e^{- \beta(b+p)^Tr}  e^{ 2\pi ib^T\theta n}\Big) \int_{\ESS^d}g_n\,d\kappa.\notag
\end{align}
In particular,
\begin{equation}\label{momentsofmu}
\phi_\nu(\iota(g_n))=\phi_\nu(U_n)
    =\Big(\sum_{b \in \NN^k} e^{-\beta b^Tr}e^{2\pi ib^T\theta n}\Big) \int_{\ESS^d} g_n\,d\kappa.
\end{equation}
Thus
\[
\phi_\nu(V_p U_n V_q^*) = \delta_{p,q} e^{- \beta p^Tr}\phi_\nu(U_n),
\]
and $\phi_\nu$ is a KMS$_\beta$ state by Proposition~\ref{prop:KMS formula}.

From~\eqref{formphinu}, we have
\begin{align*}
    \phi_\nu\big(U_n)
   &=\sum_{b \in \NN^k} e^{- \beta b^Tr } \int_{\ESS^d}  e^{2\pi i (x+b^T\theta)^Tn}\, d\kappa(x)\\
  &=\sum_{b \in \NN^k} e^{- \beta b^Tr} \int_{\ESS^d}  g_n\circ R_{\theta^Tb}\, d\kappa\\
&= \int_{\ESS^d} g_n \, d\Big(\sum_{b\in\NN^k} e^{-\beta b^Tr}R_{\theta^Tb*} (\kappa)\Big),
\end{align*}
which by Lemma~\ref{power series of Rs} is $\int_{\ESS^d} g_n\,d\nu$. Thus
\begin{equation*}\phi_\nu(V_p \iota(g_n) V_q^*)= \delta_{p,q} e^{- \beta p^Tr}\phi_\nu(\iota(g_n))=\delta_{p,q} e^{- \beta p^Tr}\int_{\ESS^d} g_n\,d\nu.
\end{equation*}
 Since  $C(\ESS^d)=\clsp\{g_n: n\in \ZZ^d\}$, \eqref{KMS from mu m} follows from~\eqref{momentsofmu} and the linearity and continuity of $\phi_\nu$.

(\ref{it:isomorphism}) We first observe that both maps $\kappa\mapsto \nu_\kappa$ and
$\nu\mapsto \phi_\nu$ are affine, and hence so is the composition. To see that the
composition is surjective, we take a KMS$_\beta$ state $\phi$, restrict it to the range
of $\iota$ to get a measure $\nu$, and take
\[
\kappa=\prod_{j=1}^k\big(\id-e^{-\beta r_j}R_{\theta_j*}\big)(\nu).
\]
Then the formula~\eqref{eq:KMS char} implies that $\phi$ and $\phi_{\nu_\kappa}$ agree on
the elements $V_p\iota(f)V_q^*$, and hence by linearity and continuity on all of
$B_\theta$. Thus $\phi=\phi_{\nu_\kappa}$. The procedure which sends $\phi$ to $\kappa$
is weak* continuous and inverts $\kappa\mapsto \phi_{\nu_\kappa}$. Thus it is a
continuous bijection of one compact Hausdorff space onto another, and is therefore a
homeomorphism. Thus so is the inverse $\kappa\mapsto \phi_{\nu_\kappa}$.
\end{proof}

\section{The subinvariance relation for the direct limit}\label{ctssubinv}

We now return to the set-up in which the dynamics $\alpha$ on the direct limit $B_\infty$
is given by a sequence $\{r^m\}$.

Suppose that $\phi$ is a KMS$_\beta$ state of $(B_\infty,\alpha)$ and $\nu_m$ are the
measures on $\ESS^d$ that implement the restrictions of $\phi\circ \pi_{m,\infty}$ to
$C(\ESS^d)\subset B_m$. Since the embeddings $\pi_m$ are all unital, so are the
$\pi_{m,\infty}$. Thus for each $m$, the restriction $\phi\circ \pi_{m,\infty}$ is a
KMS$_1$ state of $(B_m,\alpha^{r_m})$, and hence is given by a probability measure
$\nu_m$ which satisfies the subinvariance relations for $\theta=\theta_m$
in~\eqref{eq:subinv} parametrised by subsets $F$ of $\{1,\dots,k\}$. But here, since
$\phi\circ\pi_{m,\infty}=\phi\circ\pi_{m+l,\infty}\circ\pi_{m,m+l}$ for $l\in\NN$, the
measure $\nu_m$ satisfies a sequence of subinvariance relations parametrised by $l$ as
well as $F$. Our first main result says that these can be combined into one master
subinvariance relation with real parameters $s\in [0,\infty)^k$.

We now describe our continuously parametrised subinvariance relation. For $k = 1$ this
follows from \cite[Definition~6.7 and Theorem~6.9]{BHS}.

\begin{thm}\label{subinv}
Suppose that $\phi$ is a KMS$_\beta$ state on $(B_\infty,\alpha)$ and $m\in \NN$. We
write $\iota_m$ for the inclusion of $C(\ESS^d)$ in $B_m$, and then
\[
\iota_m(C(\ESS^d))=\clsp\{U_{m,n}:n\in \NN^d\}.
\]
Let $\nu_m$ be the probability measure on $\ESS^d$ such that
\begin{equation}\label{defnum} \phi\circ\pi_{m,\infty}(\iota_m(f))=\int_{\ESS^d}
f\,d\nu_m\quad \text{for $f\in C(\ESS^d)$.}
\end{equation}
Write $\theta_{m,j}$ for the $j$th row of the matrix $\theta_m$. Then  for every $s\in
[0,\infty)^k$, we have
\begin{equation}\label{eq:ctssubinv}
\prod_{j=1}^k\big(\id -e^{-\beta s_jr^m_j}R_{s_j\theta_{m,j}^T*}\big)(\nu_m)\geq 0.
\end{equation}
\end{thm}

We prove Theorem~\ref{subinv} at the end of this section. We first need two preliminary
results.

The homomorphism $\pi_m:B_m\to B_{m+1}$ maps $\iota_m(C(\ESS^d))$ into
$\iota_{m+1}(C(\ESS^d))$. When we view $\iota_m(C(\ESS^d))$ as $\clsp\{U_{m,n}\}$, the
homomorphism  $\pi_m$ is characterised by
\[
\pi_m(U_{m,n})=U_{m+1,E_mn};
\]
when we view $\iota_m(C(\ESS^d))$ as $\{\iota_m(f):f\in C(\ESS^d)\}$, $\pi_m$ is induced by
the covering map $E^T_m:\ESS^d\to \ESS^d$, and hence we have
$\pi_m(\iota_m(f))=\iota_{m+1}(f\circ E^T_m)$. In particular,  $\pi_m|_{C(\ESS^d)}$ is
$(E^T_m)^*:f\mapsto f\circ E^T_m$. The corresponding map on measures is given by
$E^T_{m*}$:
\[
\int_{\ESS^d} f\,d\pi_m(\nu)=\int_{\ESS^d} f\,dE^T_{m*}(\nu)=\int_{\ESS^d}(f\circ E^T_m)\,d\nu.
\]

\begin{lemma}\label{backinfluence}
Suppose that $\phi$ is a KMS$_\beta$ state on $(B_\infty,\alpha)$. For $m\in \NN$, let
$\nu_m$ be the probability measure on $\ESS^d$ satisfying~\eqref{defnum}. Then for every
finite subset $F$ of $\NN^k$ such that $p\wedge q=0$ for all $p\not= q\in F$, we have
\[
\prod_{p\in F}\big(\id-e^{-\beta(D_m^{-1}p)^Tr^m}R_{\theta_m^TD_m^{-1}p*}\big)(\nu_m)\geq 0.
\]
\end{lemma}

\begin{proof}
We apply Proposition~\ref{prp:KMS char}(\ref{it:subinvariance}) to the state
$\phi\circ\pi_{m+1,\infty}$ of $(B_{m+1},\alpha^{r^{m+1}})$. We deduce that
\begin{equation}\label{num+isubinv}
\prod_{p\in F}\big(\id-e^{-\beta p^Tr^{m+1}}R_{\theta_{m+1}^Tp*}\big)(\nu_{m+1})\geq 0.
\end{equation}
To convert this to a statement about $\nu_m$, we want to apply $E^T_{m*}$ to the
left-hand side. We first observe that
\begin{align}\label{commrel}
E^T_m\circ R_{\theta_{m+1}^Tp}(x)&=E_m^Tx-E_m^T\theta_{m+1}^Tp\\
&=E^T_mx-\theta_m^TD_m^{-1}p\qquad \text{using~\eqref{relatetheta}}\notag\\
&=R_{\theta_m^TD_m^{-1}p}\circ E^T_m(x).\notag
\end{align}
Since $E^T_{m*}$ preserves positivity and $h\mapsto h_*$ is covariant with respect to
composition, \eqref{num+isubinv} implies that
\begin{align*}
0&\leq E^T_{m*}\Big(\prod_{p\in F}\big(\id-e^{-\beta p^Tr^{m+1}}R_{\theta_{m+1}^Tp*}\big)(\nu_{m+1})\Big)\\
&=\Big(\prod_{p\in F}\big(\id-e^{-\beta p^Tr^{m+1}}R_{\theta_m^TD_m^{-1}p*}\big)\Big)\circ E^T_{m*}(\nu_{m+1})\quad\text{using~\eqref{commrel}}\\
&=\prod_{p\in F}\big(\id-e^{-\beta p^T(D_m)^{-1}r^m}R_{\theta_m^TD_m^{-1}p*}\big)(\nu_m)\qquad\text{using~\eqref{relater}}\\
&=\prod_{p\in F}\big(\id-e^{-\beta(D_m^{-1}p)^Tr^m}R_{\theta_m^TD_m^{-1}p*}\big)(\nu_m).\qedhere
\end{align*}
\end{proof}

For a positive integer $l$, we can apply the argument of Lemma~\ref{backinfluence} to the
embedding $\pi_{m,m+l}$ of $B_m$ in $B_{m+l}$. This amounts to replacing the matrix $D_m$
with $D_{m,m+l}:=D_{m+l-1}D_{m+l-2}\cdots D_{m+1}D_m$, $E_m$ with a similarly defined
$E_{m,m+l}$, $\theta_{m+1}$ with $\theta_{m+l}$, and $r^{m+1}$ with $r^{m+l}$. We obtain:

\begin{cor}\label{subinvl}
Suppose that $\phi$ is a KMS$_\beta$ state on $(B_\infty,\alpha)$, and $\nu_m$ is the
probability measure satisfying~\eqref{defnum}. Then for every positive integer $l$ and
for every finite subset $F$ of $\NN^k$ such that $p\wedge q=0$ for all $p\not= q\in F$,
we have
\[
\prod_{p\in F}\big(\id-e^{-\beta(D_{m,m+l}^{-1}p)^Tr^m}R_{\theta_m^TD_{m,m+l}^{-1}p*}\big)(\nu_m)\geq 0.
\]
\end{cor}

\begin{proof}[Proof of Theorem~\ref{subinv}]
For each $l\geq 0$ and $p\in \NN^k$, we can apply Corollary~\ref{subinvl} to the finite
subset $F_p:=\{p_je_j:1\leq j\leq k\}$ of $\NN^k$. This gives us the subinvariance
relation
\begin{equation}\label{F=F_p}
\prod_{j=1}^k\big(\id -e^{-\beta(D_{m,m+l}^{-1}p_je_j)^Tr^m}R_{\theta_m^TD_{m,m+l}^{-1}p_je_j*}\big)(\nu_m)\geq 0.
\end{equation}
Each factor in the left-hand side $L$ of~\eqref{F=F_p} has the form $\id-e^{-s}R_{v*}$.
Since $(e^{-s}R_{v*})(e^{-t}R_{w*})=e^{-(s+t)}R_{(v+w)*}$, the product
$(\id-e^{-s}R_{v*})(\id-e^{-t}R_{w*})$ of two such terms collapses to
\[
\id-e^{-s}R_{v*}-e^{-t}R_{w*}+e^{-(s+t)}R_{(v+w)*}.
\]
Thus we can expand
\[
L= \id + \sum_{\emptyset \not= G\subset\{1,\dots, k\}}(-1)^{|G|}e^{-\beta(D_{m,m+l}^{-1}p_G)^Tr^m}R_{\theta_m^TD_{m,m+l}^{-1}p_G*}(\nu_m),
\]
where $p_G:=\sum_{j\in G}p_je_j$.

For each fixed $f\in C(\ESS^d)$ and $\nu\in P(\ESS^d)$, the function $s\mapsto \int
R_s(f)\,d\nu$ on $\RR^k$ is continuous, being the composition of the norm-continuous map
$s\mapsto R_s(f)$ and the bounded functional given by integration against $\nu$. We now
consider a positive function $f$ in $C(\ESS^d)$: we write $f\in C(\ESS^d)_+$.  For $s\in
[0,\infty)^k$ and $G\subset\{1,\dots,k\}$, we write $s_G=\sum_{j\in G}s_je_j$. Then
\[
g_G:s\mapsto \int f\,d(e^{-\beta s_G^Tr^m}R_{\theta_m^Ts_G*})(\nu_m)
\]
 is continuous, and so is the linear combination
\[
L(s):=\int f\,d\Big(\sum_{G\subset\{1,\dots, k\}}(-1)^{|G|}e^{-\beta s_G^Tr^m}R_{\theta_m^Ts_G*}\Big)(\nu_m).
\]
The subinvariance relation~\eqref{F=F_p} says that $L(s)\geq 0$ for all $s$ of the form
$D_{m,m+l}p$ for $l\geq 0$ and $p\in \NN^k$.

Since each of the matrices $D_m$ is diagonal with entries $d_{m,j}$, say, at least $2$,
we have
\[
D_{m,m+l}^{-1}p_je_j=\Big(\prod_{n=0}^{l-1}d_{m+n,j}^{-1}\Big)p_je_j.
\]
Since $d_{n,j}\geq 2$ for all $n$ and $j$, the rational numbers of the form
$\big(\prod_{n=0}^{l-1}d_{m+n,j}^{-1}\big)p_j$ are dense in $[0,\infty)$. Thus the
vectors $s$ for which $L(s)\geq 0$ form a dense subset of $[0,\infty)^k$, and the
continuity of $L$ implies that $L(s)\geq 0$ for all $s\in [0,\infty)^k$. A measure $\nu$
which has $\int f\,d\nu\geq 0$ for all $f\in C(\ESS^d)_+$ is a positive measure, and this
is what we had to prove.
\end{proof}

\section{The solution of the subinvariance relation}

We now describe the solutions to the subinvariance relation~\eqref{eq:ctssubinv}. We
observe that the formula on the right of~\eqref{measure nu of mu} below is the Laplace
transform of a periodic function, and as such is given by an integral over a finite
rectangle. This observation motivated our calculations, but in the end we found it easier
to work with the trigonometric polynomials $x\mapsto e^{2\pi i n x}$.

\begin{thm}\label{thm:sub-inv}
Let $\theta\in M_{k,d}(\ESS^d)$, $\beta\in (0,\infty)$ and $r=(r_j)\in (0,\infty)^k$.
Denote the $j$th row of $\theta$ by $\theta_j$.
\begin{enumerate}
\item\label{mutonu}  For each $\mu\in M(\ESS^d)$,  there is a nonnegative measure
    $\nu_\mu\in M(\ESS^d)$ such that
\begin{equation}\label{measure nu of mu}
\int_{\ESS^d} f\,d\nu_\mu=\int_{[0,\infty)^k}e^{ -\beta w^Tr}\int_{\ESS^d}f(x+\theta^Tw)\,d\mu(x)\,dw \quad\text{for } f\in C(\ESS^d),
\end{equation}
and $\nu_\mu$ has total mass $\|\mu\|\prod_{j=1 }^k (\beta r_j)^{-1}$. The  measure
$\nu=\nu_\mu$  satisfies the subinvariance relation
\begin{equation}\label{eq:ref sub-inv}
 \prod_{j=1}^k\big(\id - e^{- \beta s_jr_j} R_{s_j\theta_j^T*}\big)(\nu) \ge 0\quad\text {for $s\in [0,\infty)^k$.}
\end{equation}
\smallskip

\item For each $\nu$ satisfying the subinvariance relation~\eqref{eq:ref sub-inv},
    there is a measure $\mu_\nu\in M(\ESS^d)$ such that
\begin{align}\label{measure mu of nu}
\int_{\ESS^d} &f\, d\mu_\nu=
\lim_{s_k\rightarrow 0^+}\cdots \lim_{s_1\rightarrow 0^+}\frac{1}{s_k\cdots s_1}\int_{\ESS^d} f\, d \big(\textstyle{\prod_{j=1 }^k(\id - e^{-\beta s_jr_j} R_{s_j\theta_j^T*}})\big)(\nu)
\end{align}
 for $f\in C(\ESS^d)$, and $\mu_\nu$ has total mass $\|\nu\|\prod_{j=1 }^k (\beta
 r_j)$.
\smallskip

\item\label{item-c-thm:sub-inv} The map $\mu\mapsto \nu_\mu$ is an affine homeomorphism of $M(\ESS^d)$ onto the
    simplex of measures satisfying the subinvariance relation~\eqref{eq:ref sub-inv},
    and the inverse takes $\nu$ to $\mu_\nu$.
\end{enumerate}
\end{thm}

\begin{remark}\label{extesubinv}
A measure $\nu$ that satisfies the subinvariance relation~\eqref{eq:ref sub-inv} also
satisfies the analogous relation involving $\prod_{j\in J}(\id-e^{-\beta s_j
r_j}R_{s_j\theta_j^T*})$ for any subset $J$ of $\{1,\dots,k\}$. To see this, observe that
for any vector $y\in [0,\infty)^d$, $R_y$ is an isometric positivity-preserving linear
operator on $C(\ESS^d)$. Hence so are $R_{y*}$ and $e^{-\beta s_j r_j} R_{y*}$. Since the
numbers $-\beta r_j$ are negative, the series $\sum_{n=0}^\infty e^{-\beta
s_jr_jn}R_{y*}^n$ converges in norm in the Banach space of bounded linear operators on
$M(\ESS^d)$ to an inverse for $\id-e^{-\beta s_jr_j}R_{y*}$. Hence applying this inverse
allows us to remove factors from the subinvariance relation without losing positivity.
\end{remark}

\begin{remark}[Reality check]\label{reconciliation}
We reassure ourselves that the description of subinvariant measures in
Theorem~\ref{thm:sub-inv} is consistent with the description in \cite[Theorem~7.1]{BHS}.
There $d=k=1$, and they describe the simplex of subinvariant probability measures by
specifying the extreme points of the simplex.

We recall that the matrices $D_m\in M_1(\NN)=\NN$ and $E_m\in M_1(\NN)$ are all the same
integer $N\geq 2$, and the sequence $\theta_m$ then satisfies $N^2\theta_{m+1}=\theta_m$.
In terms of our generators, the dynamics  $\alpha:\RR\to \Aut B_m$ in \cite{BHS} is given
by
\[
\alpha_t(V_{m,p}U_{m,n}V_{m,q}^*)=e^{it(p-q)N^{-m}}V_{m,p}U_{m,n}V_{m,q}^*
\]
(see \cite[Proposition~6.3]{BHS}), which is our $\alpha^{r^m}$ with $r^m=N^{-m}$. We are
interested in KMS$_\beta$ states, so the subinvariant probability measures for
$(B_m,\alpha)$ are those in the set denoted $\Omega_{\sub}^r$ for $r=\beta
N^{-m}\theta_m^{-1}=\beta r^m\theta_m^{-1}$ (see \cite[Notation~6.8]{BHS}\footnote{The
displayed equation there is meant to say this, as opposed to $r=\beta N^{-m}\theta_m$,
which is the way we first read it.}).

Since the calculation in \cite{BHS} is about extreme points, we start with a point mass
$\delta_y\in P(\varprojlim \ESS)$. Then $(E^T_m)_*\delta_y$ is the point mass
$\mu_m=\delta_{y_m}$, where $y_m$ is obtained by realising $y$ as a sequence $\{y_m\}$
satisfying $Ny_{m+1}=y_m$. Then the measure $\nu_{\mu_m}$ in
Theorem~\ref{thm:sub-inv}\eqref{mutonu} is given by
\begin{align*}
\int f\, d\nu_{\mu_m}&=\int_0^\infty e^{-\beta wr^m}\int_0^1 f(x+\theta_mw)\,d\mu_m(x)\,dw\\
&=\int_0^\infty e^{-\beta wr^m}f(y_m+\theta_mw)\,dw.
\end{align*}
For $f(x)=e^{2\pi inx}$, we get
\[
\int f\,d\nu_{\mu_m}=e^{2\pi iny_m}\int_0^\infty e^{-\beta wr^m}e^{2\pi in\theta_mw}\,dw,
\]
and a change of variables gives
\begin{align*}
\int f\,d\nu_{\mu_m}&=e^{2\pi iny_m}\int_0^\infty e^{-\beta \theta_m^{-1}vr^m}e^{2\pi inv}\theta_m^{-1}\,dv\\
&=e^{2\pi iny_m}\theta_m^{-1}\int_0^\infty e^{-(\beta r^m\theta_m^{-1})v}e^{2\pi inv}\,dv.
\end{align*}
Now we recognise the integral as the Laplace transform of the periodic function $x\mapsto
e^{2\pi inx}$, and hence
\begin{equation}\label{reconcile}
\int f\,d\nu_{\mu_m}=e^{2\pi iny_m}\theta_m^{-1}\frac{1}{1-e^{-\beta r^m\theta_m^{-1}}}\int_0^1e^{-(\beta r^m\theta_m^{-1})v}e^{2\pi i nv}\,dv.
\end{equation}
In the notation of \cite{BHS}, we set $r:=\beta r^m\theta_m^{-1}$, and
rewrite~\eqref{reconcile} as
\begin{align*}
\int f\,d\nu_{\mu_m}&=e^{2\pi iny_m}\beta^{-1}N^m\int_0^1\frac{r}{1-e^{-r}}e^{-rv}e^{2\pi i nv}\,dv\\
&=\beta^{-1}N^m \int_0^1 e^{2\pi inv}\,d(R_{y_m})_*(m_r)(v).
\end{align*}

This shows that the measure $\nu_{\mu_m}$ is a multiple of the measure $(R_{y_m})_*(m_r)$
appearing in \cite[Theorem~7.1]{BHS}. We are off by the scalar $\beta^{-1}N^m$ because
that theorem is about the simplex of subinvariant \emph{probability} measures, and the
measures $\nu_{\mu}$ in Theorem~\ref{thm:sub-inv} have total mass $(\beta
r^m)^{-1}=\beta^{-1}N^m$.
\end{remark}

For the proof of Theorem~\ref{thm:sub-inv}(a), we need the following lemma, which is
known to probabilists as the \emph{inclusion-exclusion principle}. We couldn't find a
good reference for this measure-theoretic version,  but fortunately it is relatively easy
to prove by induction on the number $k$ of subsets.

\begin{lemma}\label{in-ex}
Suppose that $\lambda$ is a finite measure on a space $X$ and $\{S_j:1\leq j\leq k\}$ is
a finite collection of measurable subsets of $X$. For each subset $G$ of $\{1,\dots
,k\}$, we set $S_G:=\bigcap_{j\in G}S_j$. Then
\[
\lambda\big({\textstyle{\bigcup_{j=1}^k S_j}}\big)=\sum_{\emptyset\not=G\subset\{1,\dots,k\}} (-1)^{|G|-1}\lambda(S_G).
\]
\end{lemma}

\begin{proof}[Proof of Theorem~\ref{thm:sub-inv}(a)]
We first claim that there is a positive functional $I$ on $C(\ESS^d)$ such that $I(f)$ is
given by the right-hand side of~\eqref{measure nu of mu}. Indeed, the estimate
\[
\Big|\int_{[0,\infty)^k}e^{ -\beta w^Tr}\int_{\ESS^d}f(x+\theta^Tw)\,d\mu(x)\,dw\Big|\leq \int_{[0,\infty)^k} e^{-\beta w^Tr}\int_{\ESS^d}\|f\|_\infty d\mu(x)dw,
\]
shows that  the right-hand side of~\eqref{measure nu of mu} determines a bounded function
$I:C(\ESS^d)\to \CC$. This function $I$ is linear because the integral is linear, and
$f\geq 0$ implies $I(f)\geq 0$ because all the integrands in~\eqref{measure nu of mu} are
non-negative. Thus there is a finite nonnegative measure $\nu_\mu$ satisfying
\eqref{measure nu of mu}. The norm of the integral is given by the total mass of the
measure $\nu_\mu$, which is
\[
\int_{\ESS^d}1\,d\nu_\mu=\int_{[0, \infty)^k} e^{-\beta w^Tr}\|\mu\|\,dw.
\]
To compute the exact value of the integral, observe that
\[
e^{-\beta w^Tr}=e^{-\beta\sum_{j=1}^kw_jr_j}=\prod_{j=1}^k e^{-\beta w_jr_j}.
\]
Thus
\begin{align*}
\|\nu_\mu\|&=\|\mu\|\int_{[0,\infty)^k}\prod_{j=1}^k e^{-\beta w_jr_j}dw=\|\mu\|\prod_{j=1}^k\int_0^\infty e^{-\beta w_jr_j} dw_j=\|\mu\|\prod_{j=1}^k(\beta r_j)^{-1}.
\end{align*}
This proves the assertions in the first sentence of part (a).

For the subinvariance relation, we fix $f\in C(\ESS^d)_+$, and aim to prove that
\begin{equation*}
\int_{\ESS^d}f\,d\Big(\prod_{j=1}^k\big(\id - e^{- \beta s_jr_j} R_{s_j\theta_j*}\big)(\nu_\mu)\Big) \geq 0.
\end{equation*}
As in the proof of Theorem~\ref{subinv}, we write
\[
\prod_{j=1}^k\big(\id - e^{- \beta s_jr_j} R_{s_j\theta_j*}\big)
=\id+\sum_{\emptyset\not=G\subset\{1\leq j\leq k\}}(-1)^{|G|}e^{-\beta s_G^Tr}R_{\theta^Ts_G*},
\]
with $s_G:=\sum_{j\in G}s_je_j$. For $j\leq k$ we define $S_j=\{v\in [0,\infty):v_j\geq
s_j\}$, and $S_G := \bigcap_{j\in G}S_{j}$. Then
\begin{align}
\int_{\ESS^d}f\,d\big(e^{- \beta s_G^Tr}&R_{\theta^Ts_G*}\big)(\nu_\mu)
=\int_{\ESS^d}e^{- \beta s_G^Tr}\big(f\circ R_{\theta^Ts_G}\big)\,d\nu_\mu\label{intG}\\
&=\int_{[0,\infty)^k}e^{ -\beta w^Tr}e^{- \beta s_G^Tr}\int_{\ESS^d}f(x+\theta^Tw+\theta^Ts_G)\,d\mu(x)\,dw\notag\\
&=\int_{S_G}e^{ -\beta v^Tr}\int_{\ESS^d}f(x+\theta^Tv)\,d\mu(x)\,dv.\notag
\end{align}

Since $f$ is fixed, we can define a measure $m$ on $[0,\infty)^k$ by
\[
\int_{[0,\infty)^k} g\,dm=\int_{[0,\infty)^k}g(v)e^{ -\beta v^Tr}\int_{\ESS^d}f(x+\theta^Tv)\,d\mu(x)\,dv.
\]
Now~\eqref{intG} says that
\[
\int_{\ESS^d}f\,d\big(e^{- \beta s_G^Tr}R_{\theta^Ts_G*}\big)(\nu_\mu)
=m(S_G).
\]
Thus
\[
\int_{\ESS^d}f\,d\Big(\prod_{j=1}^k\big(\id - e^{- \beta s_jr_j} R_{s_j\theta_j*}\big)(\nu_\mu)\Big)=m([0,\infty)^k)+\sum_{\emptyset\not=G\subset\{1\leq j\leq k\}}(-1)^{|G|}m(S_G).
\]
By the inclusion-exclusion principle, this is
\[
m([0,\infty)^k)-m\big(\textstyle{\bigcup_{j=1}^k S_j}\big)=m\big({\textstyle\prod_{j=1}^k}[0,s_j)\big) \ge 0.\qedhere
\]
\end{proof}

We now move towards a proof of part (b), and for that the first problem is to prove that
the iterated limit in~\eqref{measure mu of nu} exists. We will work with $l$ satisfying
$1\leq l\leq k$, and show by induction on $l$ that the iterated limit
\[
\lim_{s_l\to 0^+}\cdots\lim_{s_1\to 0^+}
\]
exists. We will be doing some calculus, so we often assume that our test functions $f$
belong to the dense subalgebra $C^\infty(\ESS^d)$ of $C(\ESS^d)$ consisting of smooth
functions all of whose derivatives are also periodic.

We begin by establishing that, even after dividing by the numbers which are going to $0$,
the norms of the measures remain uniformly bounded.

\begin{lemma}\label{lsnorm-bounded}
Suppose that $\nu$ is a finite positive measure on $\ESS^d$ satisfying the subinvariance
relation~\eqref{eq:ref sub-inv}. Then for each $s\in (0,\infty)^k$,
\[
\lambda_s:=\frac{1}{s_ks_{k-1}\cdots s_1}\prod_{j=1}^k\big(\id - e^{- \beta s_jr_j} R_{s_j\theta_j^T*}\big)(\nu)
\]
is a positive measure with total mass
\begin{align}\label{estlambdas}
\|\lambda_s\|\leq \Big(\prod_{j=1}^k(\beta r_j)\Big)\|\nu\|.
\end{align}
\end{lemma}

\begin{proof}
The subinvariance relation implies that the measure is positive.  For the estimate on the
total mass of $\lambda_s$, we deal with the variables $s_i$ separately. So for $1\leq
l\leq k$, we set
\[
\sigma_l:=\prod_{j=l}^k\big(\id-e^{-\beta s_jr_j}R_{s_j\theta_{j}^T*}\big)(\nu),
\]
which by Remark~\ref{extesubinv} are all positive measures. We have
\begin{align*}
\int_{\ESS^d}1\,d\Big(\prod_{j=1}^k\big(\id - e^{- \beta s_jr_j} R_{s_j\theta_j^T*}\big)(\nu)\Big)
    &=\int_{\ESS^d}1\,d\sigma_1\\
    &=\int_{\ESS^d}1\,d\big((\id - e^{- \beta s_1r_1} R_{s_1\theta_1^T*})(\sigma_2)\big)\notag\\
    &=\int_{\ESS^d}1\circ\big(\id - e^{- \beta s_1r_1} R_{s_1\theta_1^T}\big)\,d\sigma_2\notag\\
    &=\int_{\ESS^d} (1-e^{-\beta s_1r_1})\,d\sigma_2.
\end{align*}
So for all $s_1>0$,
\begin{equation}\label{estsigma1}
\frac{1}{s_1}\int_{\ESS^d}1\,d\Big(\prod_{j=1}^k\big(\id - e^{- \beta s_jr_j} R_{s_j\theta_j^T*}\big)(\nu)\Big)
    = \int_{\ESS^d}\frac{1-e^{-\beta s_1r_1}}{s_1}\,d\sigma_2.
\end{equation}
The integrand here is
\[
\frac{1-e^{-\beta s_1r_1}}{s_1}=\frac{f(0)-f(s_1)}{s_1}\quad\text{for $f(s_1):=e^{-\beta s_1r_1}$.}
\]
Hence for each fixed $s_1>0$, the mean value theorem implies that there exists $c\in
(0,s_1)$ such that
\[
\frac{1-e^{-\beta s_1r_1}}{s_1}=-f'(c)=-(-\beta r_1e^{-\beta cr_1}),
\]
which is a positive number less than $\beta r_1$. Thus~\eqref{estsigma1} is at most
$\beta r_1\|\sigma_2\|$.

Now we repeat this argument, first to see that
\[
\frac{1}{s_2}\int_{\ESS^d}1\,d\big(\id - e^{- \beta s_2r_2} R_{s_2\theta_2^T*}\big)(\sigma_3)
\]
has mass at most $\beta r_2\|\sigma_3\|$. After $k-2$ more steps, we arrive at  the
estimate~\eqref{estlambdas}.
\end{proof}

\begin{lemma}\label{|L|=1}
Suppose that $1\leq j\leq k$ and that $\lambda\in M(\ESS^d)$ satisfies
\[
\big(\id-e^{-\beta sr_j}R_{s\theta_j^T*}\big)(\lambda)\geq 0\quad\text{for all $s>0.$}
\]
Then for all $f\in C^\infty(\ESS^d)$, we have
\begin{equation}\label{l=1}
\lim_{s \to 0^+}\Big(
\frac{1}{s}\int_{\ESS^d}f\,d(\id-e^{-\beta sr_j}R_{s\theta_j^T*})(\lambda)\Big) = \beta r_j\int_{\ESS^d}f\,d\lambda -\int_{\ESS^d}\theta_j^T(\nabla f)\,d\lambda.
\end{equation}
\end{lemma}

\begin{proof}
Let $g : \ESS^d \times \RR \to \CC$ be the function $g(x,s)=e^{-\beta
sr_j}f(x+s\theta_j^T)$. The term on the left of~\eqref{l=1} can be rewritten as
\[
\frac{1}{s}\int_{\ESS^d}\big(f(x)-e^{-\beta sr_j}f(x+s\theta_j^T)\big)\,d\lambda(x)=\frac{1}{s}\int_{\ESS^d}(g(x,0)-g(x,s))\,d\lambda(x).
\]
So we want to show that the function $G$ defined by
$G(s):=\int_{\ESS^d}g(x,s)\,d\lambda(x)$ is differentiable at $0$ with $-G'(0)$ equal to
the right-hand side of~\eqref{l=1}.

We compute
\[
\frac{\partial g}{\partial s}(x,s)=-\beta r_je^{-\beta sr_j}f(x+s\theta_j^T)+ e^{-\beta sr_j}\theta_j^T\nabla f(x+s\theta_j^T).
\]
The Cauchy-Schwarz inequality for the inner product $\theta_j^T(\nabla
f)=(\theta_j\,|\,\nabla f)$ then gives
\begin{equation}
\Big|\frac{\partial g}{\partial s}(x,s)\Big|
\leq \beta r_j\|f\|_\infty +\|\theta_j^T\|_2\|\nabla f(x+s\theta_j^T)\|_2.
\end{equation}
The right-hand side is uniformly bounded on $\ESS^d$, and hence there is an integrable
function on $\ESS^d$ that dominates the right-hand side for all $s\in [0,1]$, say. Thus we
can differentiate under the integral sign, using Theorem~2.27 of \cite{Fo}, for example.
We deduce that $G$ is differentiable on $[0,1]$ with derivative
\[
G'(s)=\int_{\ESS^d}\big(-\beta r_je^{-\beta sr_j}f(x+s\theta_j^T)+ e^{-\beta sr_j}\theta_j^T\nabla f(x+s\theta_j^T)\big)\,d\lambda(x).
\]
Taking $s=0$ gives the negative of the right-hand side of~\eqref{l=1}, as required.
\end{proof}

Our next step is the inductive argument, which is quite a complicated one. As a point of
notation, for each tuple $I=\{i_1,\dots,i_m\}$ with entries in $\{1,2,\dots, k\}$, and
for $f\in C^\infty(\ESS^d)$, we write $|I|:=m$ and $D_If$ for the partial derivative
\[
D_If:=\frac{\partial^mf}{\partial x_{i_1}\partial x_{i_2}\cdots\partial x_{i_m}}.
\]

\begin{lemma}\label{limitsexist}
Suppose that $\nu$ is a positive measure on $\ESS^d$ satisfying the subinvariance
relation~\eqref{eq:ref sub-inv}. Let $1 \le l \le k$.
\begin{itemize}
\item[(a)] The iterated limit
\begin{align*}
\lim_{s_l\rightarrow 0^+}\cdots \lim_{s_1\rightarrow 0^+}\frac{1}{s_l\cdots s_1}\int_{\ESS^d} f\, d \big(\textstyle{\prod_{j=1 }^l}(\id &- e^{-\beta s_jr_j} R_{s_j\theta_j^T*})\big)(\nu)
\end{align*}
exists for all $f\in C(\ESS^d)$.

\item[(b)] Write
\[
\Sigma_l:={\textstyle\bigcup_{n=0}^l}\{1,\dots,k\}^n.
\]
Then there are real scalars $\{K^l_I:I\in \Sigma_l\}$ such that
$K^l_\emptyset=\prod_{j=1}^l(\beta r_{j})$, and: for every $f\in C^\infty(\ESS^d)$ and
for every measure $\nu$ on $\ESS^d$ satisfying the subinvariance
relation~\eqref{eq:ref sub-inv}, the limit in \textnormal{(a)}  is
\begin{equation}\label{formformunu}
\int_{\ESS^d}\Big(\sum_{I\in\Sigma_l}K^l_ID_If\Big)\,d\nu.
\end{equation}
\end{itemize}
\end{lemma}

\begin{proof}
We prove by induction on $l$ that the limit in (a) exists for every $f\in
C^\infty(\ESS^d)$, and that there exist the scalars $K^l_I$. Then, since we know from
Lemma~\ref{lsnorm-bounded} that the measures $\lambda_s$ are norm-bounded by
$\big(\prod_{j=1}^k(\beta r_j)\big)\|\nu\|$ and $C^\infty(\ESS^d)$ is norm dense in
$C(\ESS^d)$, we get convergence in (a) also for $f\in C(\ESS^d)$.

When $l=1$, the index set $\Sigma_1$ consists of the empty set $\emptyset$ and the
one-point sets $\{j\}$. Lemma~\ref{|L|=1} implies that $K^1_\emptyset=\beta r_1$ and
$K^1_{\{j\}}= \theta_1^Te_j=\theta_{1j}$.

We fix $l$ between $1$ and $k-1$, and suppose as our inductive hypothesis that for every
measure $\lambda$ such that
\begin{equation}\label{subinvforl}
\prod_{j=1}^l\big(\id-e^{-\beta s_jr_j}R_{s_j\theta_{j}^T*}\big)(\lambda)\geq 0\quad\text{for all $s\in [0,\infty)^k$,}
\end{equation}
we have such scalars $\{K_I^l\}$ parametrised by $I\in \Sigma_l$. We now have to start
with a measure $\kappa$ that satisfies
\begin{equation}\label{subinvl+1}\prod_{j=1}^{l+1}\big(\id-e^{-\beta s_jr_j}R_{\theta_{j}^T*}\big)(\kappa)\geq 0\quad\text{for all $s\in [0,\infty)^k$,}
\end{equation}
and find suitable scalars $K^{l+1}_I$.

We define
\[
\lambda:=\big(\id-e^{-\beta s_{l+1}r_{l+1}}R_{s_{l+1}\theta_{l+1}^T*}\big)(\kappa).
\]
Remark~\ref{extesubinv} reassures us that $\lambda$ is another positive measure,
and~\eqref{subinvl+1} implies that it satisfies~\eqref{subinvforl}. The induction
hypothesis gives
\begin{align*}
L(s_{l+1}):=\lim_{s_{l}\rightarrow 0^+}\cdots \lim_{s_{1}\rightarrow 0^+}\frac{1}{s_{l+1}\cdots s_{1}}\int_{\ESS^d} f\, d&\Big(\prod_{j=1 }^{l+1}\big(\id - e^{-\beta s_jr_j } R_{s_j\theta^T_j*}\big)\Big)(\kappa)\\
&=\frac{1}{s_{l+1}}\bigg(\int_{\ESS^d}\Big(\sum_{I\in\Sigma_l}K_I^lD_If\Big)\,d\lambda\bigg).
\end{align*}
Lemma~\ref{|L|=1} implies that
\begin{align*}
\lim_{s_{l+1}\to 0+}L(&s_{l+1})=
\beta r_{l+1}\int_{\ESS^d}\Big(\sum_{I\in \Sigma_l}K_I^lD_I f\Big)\,d\lambda-\int_{\ESS^d}\theta^T_{l+1}\nabla\Big(\sum_{I\in \Sigma_l}K_I^lD_I f\Big)\, d\lambda\\
&= \beta r_{l+1}  \int_{\ESS^d}\Big(\sum_{I\in \Sigma_l}K_I^lD_I f\Big)\,d\lambda -\int_{\ESS^d} \Big(\sum_{I\in \Sigma_l}\sum_{i=1}^d K_I^l\theta_{l+1,i}\frac{\partial D_{I}f}{\partial x_i}\Big)\, d\lambda(x).
\end{align*}
To finish off the inductive step, we set $K_\emptyset^{l+1}=\beta r_{l+1} K^l_\emptyset$,
and for $I'=(I,i_{|I|+1})$ we set
\begin{align*}
K^{l+1}_{I'}=\begin{cases}
K^l_{I}\theta_{(l+1)i_{l+1}}&\text{ if } |I|=l\\
\beta r_{l+1}K^l_{I'} -K_{I}^l\theta_{l+1,i_{|I|+1}}&\text{ if } |I|<l.
\end{cases}
\end{align*}
This completes the inductive step, and hence the proof.
\end{proof}

\begin{proof}[Proof of Theorem~\ref{thm:sub-inv}(b)]
Lemma~\ref{limitsexist} shows that the limit exists for all $f\in C(\ESS^d)$, and for
$f\in C^\infty(\ESS^d)$ gives us a formula for the limit. The limit is linear in $f$,
positive when $f$ is, and is bounded by $\|f\|_\infty \big(\prod_{j=1}^k(\beta
r_j)\big)\|\nu\|$. Thus it is given by a finite positive measure $\mu_\nu$. Since the
total mass of the measure is integration against the constant function $1$, and since $1$
is smooth, the total mass is given by~\eqref{formformunu}. But since all derivatives of
$1$ are zero, the only nonzero terms are the ones on which $I=\emptyset$.  Now the
formula for $K^k_\emptyset$ implies that $\|\mu_\nu\|=\big(\prod_{j=1}^k(\beta
r_j)\big)\|\nu\|$.
\end{proof}

We now work towards the proof of Theorem~\ref{thm:sub-inv}(c). To prove that
$N:\mu\mapsto \nu_\mu$ is a bijection of the measures arising from KMS$_\beta$ states
onto the subinvariant measures, we prove that $N$ is one-to-one and that $M:\nu\mapsto
\mu_\nu$ satisfies $N\circ M(\nu)=\nu$ for all subinvariant measures~$\nu$. We then have
$N\circ (M\circ N)(\mu)=(N\circ M)\circ N(\mu) =N(\mu)$, and injectivity of $N$ implies
$(M\circ N)(\mu)=\mu$. Thus Theorem~\ref{thm:sub-inv}(c) follows from the following
proposition.

\begin{prop}\label{recovernu}
Suppose that $\nu$ is a measure on $\ESS^d$ satisfying the subinvariance
relation~\eqref{eq:ref sub-inv} and with total mass $\prod_{j=1}^k(\beta r_j)^{-1}$. Then
$\nu=\nu_{\mu_\nu}$.
\end{prop}

Suppose that $\nu$ is a subinvariant measure and $f\in C^\infty(\ESS^d)$. We need to show
that the functional defined by integrating against $\nu_{\mu_\nu}$, which is defined in
parts (a) and (b) of the theorem as
\begin{align}\label{nu-numunu}
\int_{[0,\infty)^k}e^{- \beta w^Tr}\lim_{s\to 0^+}\frac{1}{s_k\cdots s_1}\int_{\ESS^d} f(x+\theta^T w)\, d \Big(\prod_{j=1 }^k\big(\id - e^{- \beta s_jr_j } R_{s_j\theta_j^T*}\big)\Big)(\nu)(x) dw,
\end{align}
is in fact implemented by $\nu$. We will do this by peeling off the iterated limit one
variable at a time. For this, the next lemma is crucial.

\begin{lemma}\label{limit formula}
Consider a positive measure $\lambda$ on $\ESS^d$, $b\in (0,\infty)$ and $v\in \ESS^d$. For
$f\in C^\infty(\ESS^d)$, we have
\begin{align*}
\int_0^\infty e^{-bt}\lim_{s\to 0+}\frac{1}{s}\int_{\ESS^d}&f(y+tv)\,d\big((\id-e^{-bs}R_{sv*})\lambda\big)(y)\,dt\\
&=\lim_{s\to 0+}\int_0^\infty \frac{e^{-bt}}{s}\int_{\ESS^d}f(y+tv)\,d\big((\id-e^{-bs}R_{sv*})\lambda\big))(y)\,dt.
\end{align*}\end{lemma}

\begin{proof}
For $s>0$, we have
\[
\frac{1}{s}\int_{\ESS^d}f(y+tv)\,d\big((\id-e^{-bs}R_{sv*})\lambda\big)(y)=\frac{1}{s}\int_{\ESS^d}\big(f(y+tv)-e^{-bs}f(y+tv+sv)\big)\,d\lambda(y).
\]
We write this last integrand as
\begin{align*}
K(y,&s,t)=s^{-1}\big(f(y+tv)-e^{-bs}f(y+tv+sv)\big)\\
&=s^{-1}\big(f(y+tv)-e^{-bs}f(y+tv)+e^{-bs}f(y+tv) -e^{-bs}f(y+tv+sv)\big)\\
&=\frac{1-e^{-bs}}{s}f(y+tv)+e^{-bs}\frac{f(y+tv) -f(y+tv+sv)}{s}.
\end{align*}
We estimate the first summand using the mean value theorem on $e^{-bs}$, and the second
summand using the same theorem on $f$, to find
\[
|K(y,s,t)|\leq b\|f\|_\infty+\|v^T(\nabla f)\|_\infty.
\]
Thus we have
\[
\Big|\frac{e^{-bt}}{s}\int_{\ESS^d} K(y,s,t)\,d\lambda(y)\Big|\leq e^{-bt}\big(
b\|f\|_\infty+\|v^T(\nabla f)\|_\infty\big)\|\lambda\|.
\]
Now the result follows from the dominated convergence theorem for Lebesgue measure on
$[0,\infty)$ (modulo the trick of observing that it suffices to work with sequences
$s_n\to 0+$ --- see the proof of \cite[ Theorem~2.27]{Fo}).\
\end{proof}

\begin{proof}[Proof of Proposition~\ref{recovernu}]
As in the proof of Theorem~\ref{thm:sub-inv}(b), Lemma~\ref{limitsexist} implies that
there is  a positive measure $\eta$ on $\ESS^d$ such that for $g\in C^\infty(\ESS^d)$, we
have
\[
\int_{\ESS^d}g\,d\eta=\lim_{s_k,\dots, s_2\rightarrow 0^+}\frac{1}{s_k\cdots s_2}\int_{\ESS^d} g\,d \Big(\prod_{j=2 }^k\big(\id - e^{-\beta s_j r_j } R_{s_j\theta^T_j*}\big)\Big)(\nu).
\]
Since  the operators $\id - e^{-\beta s_jr_j } R_{s_j\theta^T_j*}$ commute with each
other,
\begin{align}\label{density one}
\int_{\ESS^d} g\,
d \big(\id - &e^{-\beta s_1 r_1 } R_{s_1\theta_1^T*}\big)(\eta)\\
&=\notag\lim_{s_k,\dots, s_2\rightarrow 0^+}\frac{1}{s_k\cdots s_2}\int_{\ESS^d} g\,d \Big(\prod_{j=1 }^k\big(\id - e^{-\beta s_jr_j } R_{s_j\theta^T_j*}\big)\Big)(\nu).
\end{align}

Now we need some complicated notation to implement the peeling process. First of all, we
fix $f\in C^\infty(\ESS^d)$. In an attempt to avoid an overdose of subscripts, we write
$s=(s_1,\hat s)$, $w=(w_1,\hat w)$ and $r=(r_1,\hat r)$. We also write $\hat \theta$ for
the ${k-1}\times d$ matrix obtained from $\theta$ by deleting the first row: thus
$\theta^T$ has block form $(\theta_1^T\ \hat\theta^T)$. With the new notation,
\eqref{nu-numunu} becomes
\begin{align*}
\int_{[0,\infty)^{k-1}}&e^{- \beta\hat{w}^T \hat{r} }\int_0^\infty e^{-\beta r_1 w_1}\times\\
&\times\lim_{s_1\rightarrow 0^+}\frac{1}{s_1}\int_{\ESS^d}
f(x+\hat{\theta}^T \hat{w}+w_1\theta_1^T)\,
d \big(\id - e^{-\beta r_1 w_1} R_{s_1\theta_1^T*}\big)(\eta)(x)dw_1 d\hat{w}.
\end{align*}
Now we can apply Lemma~\ref{limit formula} to the inside integrals, which gives
\begin{align}\label{6 beta formula}
\int_{\ESS^d}f\,d\nu_{\mu_\nu}=\int_{[0,\infty)^{k-1}}&e^{- \beta \hat w^T \hat{r}}\lim_{s_1\rightarrow 0^+}\frac{1}{s_1}\int_0^\infty e^{-\beta r_1 w_1}\times\\
\times\int_{\ESS^d} &f(x+\hat{\theta}^T \hat{w}+w_1\theta_1^T)\,
d \big(\id - e^{- \beta r_1 w_1} R_{s_1\theta_1^T*}\big)(\eta)(x)dw_1 d\hat{w}.\notag
\end{align}

We now consider the function $g$ on $(0,\infty)$ defined by
\[
g(s_1):=\frac{1}{s_1}\int_0^\infty e^{-\beta r_1 w_1}\\
\int_{\ESS^d} f(x+\hat{\theta}^T \hat{w}+w_1\theta_1^T)\,
d \big(\id - e^{- \beta r_1 w_1} R_{s_1\theta_1^T*}\big)(\eta)(x)dw_1.
\]
We aim to prove that $g(s_1)\to\int_{\ESS^d}f(x+\hat{\theta}^T\hat{w})\,d\eta(x)$ as
$s_1\to 0^+$. To this end, we compute:
\begin{align*}
g(s_1)=\frac{1}{s_1}\int_0^\infty\!\!e^{- \beta r_1 w_1}&\!\!\int_{\ESS^d}
f\big(x+\hat{\theta}^T \hat{w}+w_1\theta_1^T\big)\,
d\eta(x)dw_1\\
-\frac{1}{s_1}&\int_0^\infty\!\! e^{- \beta r_1(w_1+s_1)}\!\!\int_{\ESS^d}
f\big(x+\hat{\theta}^T \hat{w}+(w_1+s_1)\theta_1^T\big)\,
d\eta(x)dw_1.
\end{align*}
Changing the variable in the second integral to get an integral over $[s_1,\infty)$
gives
\[
g(s_1)=\frac{1}{s_1}\int_{0}^{s_1}\!\!e^{- \beta r_1 w_1}\int_{\ESS^d}
f\big(x+\hat{\theta}^T \hat{w}+w_1\theta_1^T\big)\,d\eta (x)dw_1.
\]
Now we have
\begin{align*}
g(s_1)-&\int_{\ESS^d}f(x+\hat{\theta}^T \hat{w})\,d\eta(x)\\
&=\frac{1}{s_1}\int_{0}^{s_1}e^{-\beta r_1w_1}\!\!\!\int_{\ESS^d}
f(x+\hat{\theta}^T \hat{w}+w_1\theta^T_1)\,d\eta (x)dw_1-\int_{\ESS^d}f(x+\hat{\theta}^T \hat{w})\,d\eta(x)\\
&=\frac{1}{s_1}\int_{0}^{s_1}\int_{\ESS^d}
\big(e^{- \beta r_1w_1}f(x+\hat{\theta}^T \hat{w}+w_1\theta_1^T)-f(x+\hat{\theta}^T \hat{w})\big)\,d\eta (x)dw_1.
\end{align*}
Since $y\mapsto e^{- \beta y^T r}f(x+\hat{\theta}^T \hat{w}+\theta^Ty) $  is uniformly
continuous, there exists $\delta$ such that
\[
0\leq w_1<\delta \Longrightarrow \big|e^{-\beta r_1 w_1 }f(x+\hat{\theta}^T \hat{w}+w_1\theta_1^T)-f(x+\hat{\theta}^T \hat{w})\big|<\frac{\epsilon}{\|\eta\|}.
\]
So for $0\leq s_1<\delta$ we have
\[
\Big|g(s_1)-\int_{\ESS^d}f(x+\hat{\theta}^T \hat{w})\,d\eta(x)\Big|\leq \frac{1}{s_1}\int_0^{s_1}\int_{\ESS^d}\frac{\epsilon}{\|\eta\|}\,d\eta(x)dw_1=\epsilon.
\]
Thus $g(s_1)\to \int_{\ESS^d}f(x+\hat{\theta}^T \hat{w})\,d\eta(x)$, as we wanted.

Putting the formula for $\lim_{s_1\rightarrow 0^+}g(s_1)$ in~\eqref{6 beta formula} gives
\begin{align*}
\int_{\ESS^d}f\,d\nu_{\mu_\nu}&=\int_{[0,\infty)^{k-1}} e^{- \beta\hat{w}^T \hat{r}}\int_{\ESS^d} f(x+\hat{\theta}^T \hat{w})\, d\eta(x)\,d\hat{w},
\end{align*}
which is the right-hand side of~\eqref{nu-numunu} with one $\lim_{s\to 0^+}$ and one
$\int_0^\infty$ removed. Repeating the argument $k-1$ times gives
 \begin{align*}
\int_{\ESS^d}f\,d\nu_{\mu_\nu}&=\int_{\ESS^d} f\, d\nu,
\end{align*}
as required.
\end{proof}

As described before Proposition~\ref{recovernu}, this completes the proof of
Theorem~\ref{thm:sub-inv}.

\section{A parametrisation of the equilibrium states}

We are now ready to describe the KMS states of our system. At the end of the section, we
will use the following theorem to prove our main result.

\begin{thm}\label{bigthm}
Consider our standard set-up, and suppose that $\beta>0$.
\begin{enumerate}
\item\label{constructKMS} Suppose that $\mu\in P\big(\varprojlim(\ESS^d, E_m^T)\big)$.
    Define measures $\mu_m\in P(\ESS^d)$ by $\mu_m=E_{m,\infty*}^T(\mu)$ and take
    $\nu_{\mu_m}$ to be the subinvariant measure on $\ESS^d$ obtained by applying
    Theorem~\ref{thm:sub-inv} to the measure $\mu_m$. Then there is a KMS$_\beta$
    state $\psi_{\mu}$ of $(B_\infty,\alpha)$ such that
\begin{equation}\label{defpsimu}
\psi_\mu(V_{m,p} U_{m,n} V_{m,q}^*)=\delta_{p,q}e^{-\beta p^Tr^m}\prod_{j=1}^k(\beta r^m_j)\int_{\ESS^d} e^{2\pi i x^Tn}\, d\nu_{\mu_m}(x).
\end{equation}
\item\label{surjectivity} The map $\mu\mapsto \psi_\mu$ is an affine homeomorphism of
    $P\big(\varprojlim(\ESS^d, E_m^T)\big)$ onto KMS$_\beta (B_\infty, \alpha)$.
\end{enumerate}
\end{thm}

To prove the theorem, we first build some maps between the spaces of subinvariant
measures. We will make use of Theorem~\ref{thm:sub-inv}, but the measures described there
are not all normalised. To ensure we are dealing with probability measures, we introduce
the numbers
\[
c_m:=\prod_{j=1}^k(\beta r^m_j) = \beta^k\big(\prod_{j=1}^kr^m_j\big)\qquad\text{ and }\qquad d_m:=\det D_m.
\]
Because $D_m$ is diagonal, Equation~\ref{relater} shows that the two sets of numbers are
related by $d_m c_{m+1}=c_m$.

In particular, the functions $f_{\beta,r^m}$ from Remark~\ref{connectAaHR} have
constant value $c_m$, and so $f_{\beta,r^{m+1}} = d_m^{-1}f_{\beta,r^m}$. Thus with
$\Sigma_{\beta,r}$ from Proposition~\ref{prp:KMS char}(\ref{it:isomorphism}), we can
define $\sigma_m:\Sigma_{\beta,r^{m+1}}\to \Sigma_{\beta,r^m}$ by
\[
\sigma_m(\nu) = d_m^{-1} E^T_{m*}(\nu).
\]

\begin{lemma}\label{normalising}
Suppose that $\mu\in P\big(\varprojlim(\ESS^d,E_m^T)\big)$, and define
$\mu_m:=E^T_{m,\infty*}(\mu)$ for $m\geq 1$. Then the measures $\nu_{\mu_m}$ given by
Theorem~\ref{thm:sub-inv} satisfy $\sigma_m(\nu_{\mu_{m+1}})=\nu_{\mu_m}$.
\end{lemma}
\begin{proof}
We take $f\in C(\ESS^d)$, and compute using~\eqref{measure nu of mu}:
\begin{align*}
\int_{\ESS^d} f&\,d\sigma_m(\nu_{\mu_{m+1}})
=d_m^{-1}\int_{\ESS^d} f\circ E_m^T\, d\nu_{\mu_{m+1}}\\
&=d_m^{-1}\int_{[0,\infty)^k}e^{-\beta w^Tr^{m+1}}\!\!\int_{\ESS^d} (f\circ E_m^T)(x+\theta_{m+1}^Tw)\,d\mu_{m+1}(x)\,dw\\
&=d_m^{-1}\int_{[0,\infty)^k}e^{-\beta w^Tr^{m+1}}\!\!\int_{\ESS^d}f(E_m^Tx+E_m^T\theta_{m+1}^Tw)\,d\mu_{m+1}(x)\,dw\\
&=d_m^{-1}\int_{[0,\infty)^k}e^{-\beta w^TD_m^{-1}r^m}\!\!\int_{\ESS^d} f(E_m^Tx+\theta_{m}^TD_m^{-1}w)\,d\mu_{m+1}(x)\,dw,
\end{align*}
where at the last step we used\footnote{This also uses that $D_m\theta_{m+1}E_m=\theta_m$
on the nail, i.e. as opposed to modulo $\ZZ$. Otherwise the difference would appear in the
last formula multiplied by the real variable $w$.} both~\eqref{relater} and
\eqref{relatetheta}. Now substituting $v=D_m^{-1}w$ in the outside integral gives
\begin{equation}\label{doubleint}
\int_{\ESS^d} f\,d\sigma_m(\nu_{\mu_{m+1}})=\int_{[0,\infty)^k}e^{-\beta v^Tr^{m+1}}\!\!\int_{\ESS^d} f(E_m^Tx+\theta_{m}^Tv)\,d\mu_{m+1}(x)\,dv.
\end{equation}
We write $s:=\theta_{m}^Tv$ and consider the translation automorphism $\tau_s$ of
$C(\ESS^d)$ defined by $\tau_s(f)(x)=f(x+s)$. Then the inside integral on the right of
\eqref{doubleint} is
\begin{align*}
\int_{\ESS^d} f(E_m^Tx+\theta_{m}^Tv)\,d\mu_{m+1}(x)&=\int_{\ESS^d}(\tau_s(f)\circ E_m^T)\,d\mu_{m+1}\\
&=\int_{\ESS^d}\tau_s(f)\,d(E_m^T)_*(\mu_{m+1})\\
&=\int_{\ESS^d}f(x+\theta_{m}^Tv)\,d\mu_m(x).
\end{align*}
Putting this back into the double integral in~\eqref{doubleint} gives the right-hand side
of~\eqref{measure nu of mu} for the measure $\mu_m$, and we deduce from~\eqref{measure nu
of mu} that
\[
\int_{\ESS^d} f\,d\sigma_m(\nu_{\mu_{m+1}})=\int_{\ESS^d}  f\,d\nu_{\mu_m}\quad\text{for all $f\in C(\ESS^d)$,}
\]
as required.
\end{proof}

\begin{proof}[Proof of Theorem~\ref{bigthm}\eqref{constructKMS}]
Since the maps
\[
E^T_{m,\infty}:\varprojlim\ESS^d\to \ESS^d
\]
are surjective, each $\mu_m$ is a probability measure on $\ESS^d$. Thus we deduce from
Theorem~\ref{thm:sub-inv} that $\nu_m:=\big(\prod_{j=1}^k(\beta r^m_j)\big)\nu_{\mu_m}$
is a probability measure satisfying the subinvariance relation~\eqref{eq:ref sub-inv}.
Thus Proposition~\ref{prp:KMS char}(\ref{it:phi-mu}) gives a KMS state $\psi_m$ of $(B_m,
\alpha^{r^m})$ such that $\psi_m(U_{m,n})=\int_{\ESS^d} e^{2\pi i x^T n}\,d\nu_m(x)$. We
now need to check that $\psi_{m+1}\circ\pi_m=\psi_m$ so that we can deduce from
\cite[Proposition~3.1]{BHS} that the $\psi_m$ combine to give a KMS$_\beta$ state of
$(B_\infty,\alpha) $.

Since we are viewing measures as functionals on $C(\ESS^d)$, the map $E_{m*}^T$ on
$M(\ESS^d)$ is induced by the continuous function $E_m^T : x\mapsto E_m^Tx$ on $\ESS^d$.
Then for $f\in C(\ESS^d)$
\begin{equation}\label{connmeasures}
\int_{\ESS^d} f\circ E_m^T(x)\,d\nu_{m+1}(x) = \int_{\ESS^d} f(x)\,dE_{m*}^T(\nu_{m+1})(x).
\end{equation}
For the functions $g_{n} \in C(\ESS^d)$ given by $g_{n}(x) = e^{2\pi in^T x}$ (so
that $\iota_m(g_{n}) = U_{m,n} \in B_{m}$), we have
\[
g_{n}\circ E_m^T(x)
    = e^{2\pi i (E_m^T x)^T n}
    = e^{2\pi i x^T E_m n}
    = g_{E_m n}(x).
\]
Substituting this on the left-hand side of~\eqref{connmeasures} gives
\begin{equation}\label{connUs}
\int_{\ESS^d} g_{E_m n}(x) \,d\nu_{m+1}(x) = \int_{\ESS^d} g_{n}(x) \,dE_{m*}^T(\nu_{m+1})(x).
\end{equation}
Using again $d_m = \det D_m$ and $c_m = \prod_{j=1}^k(\beta r^m_j)$ and
the relation $d_m c_{m+1} = c_m$, Lemma~\ref{normalising} gives
\begin{equation}\label{nus match}
\begin{split}
E^T_{m*}(\nu_{m+1})
     &= \sigma_m(d_m \nu_{m+1})
     = d_m c_{m+1} \sigma_m(\nu_{\mu_{m+1}})\\
     &= d_m c_{m+1} \nu_{\mu_m}
     = c_m \nu_{\mu_m} = \nu_m.
\end{split}
\end{equation}

Using~\eqref{connUs} at the third step, and~\eqref{nus match} at the fourth step, we now
calculate:
\begin{align*}
\psi_{m+1}(\pi_m(U_{n,m}))
    &=\psi_{m+1}(U_{m+1, E_m n})
    = \int_{\ESS^d} g_{E_mn} \,d\nu_{m+1}\\
    &=\int_{\ESS^d} g_{n}\,dE^T_{m*}(\nu_{m+1})
    =\int_{\ESS^d} g_{n}\,d\nu_m
    =\psi_m(U_{m,n}).
\end{align*}
Thus the states $\psi_m$ give an element $(\psi_m)$ of the inverse limit $\varprojlim
\text{KMS}_\beta(B_m,\alpha^{r^m})$, and surjectivity of the isomorphism in
\cite[Proposition~3.1]{BHS} gives a KMS$_\beta$ state $\psi_\mu$ of $(B_\infty,\alpha)$
such that $\psi_m=\psi_\mu\circ \pi_{m,\infty}$ for $m\geq 1$.
\end{proof}

\begin{remark}
We observe that the KMS$_\beta$ state of Theorem~\ref{bigthm}\eqref{constructKMS} is
given on $B_\infty=\clsp\{V_{m,p}U_{m,n}V_{m,q^*}\}$ by
\[
\psi_\mu(V_{m,p}U_{m,n}V_{m,q}^*)=\delta_{p,q}e^{-\beta p^Tr^m}\int_{\ESS^d} g_n\,d(\nu_{E^T_{m,\infty*}(\mu)}).
\]
\end{remark}

\begin{proof}[Proof of Theorem~\ref{bigthm}\eqref{surjectivity}]
We first prove that every KMS$_\beta$ state has the form $\psi_\mu$. So suppose that
$\phi$ is a KMS$_\beta$ state of $(B_\infty,\alpha)$. Then for each $m\geq 1$,
$\phi\circ\pi_{m,\infty}$ is a KMS$_\beta$ state of $(B_m,\alpha^{r^m})$, and hence there
are probability measures $\nu_m$ such that
\[
\phi\circ \pi_{m,\infty}(f)=\int_{\ESS^d} f\,d\nu_m\quad\text{for all $f\in C(\ESS^d)$}
\]
and $E^T_{m*}(\nu_{m+1})=\nu_m$ for all $m\geq 1$. Theorem~\ref{subinv} implies that each
$\nu_m$ satisfies the corresponding subinvariance relation. More specifically, we write
$M^{\sub}_m(\ESS^d)$  and $P^{\sub}_m(\ESS^d)$ for the set of measures and the set of probability measures satisfying~\eqref{eq:ctssubinv}. 
Then we have $\nu_m\in P^{\sub}_m(\ESS^d)$.

Once more using $d_m = \det D_m$ and $c_m = \prod_{j=1}^k(\beta r^m_j)$,
the construction of Theorem~\ref{thm:sub-inv}\eqref{mutonu} gives a function $\mu\mapsto
c_m\nu_{\mu}$ from $M(\ESS^d)$ to the simplex $M^{\sub}_m(\ESS^d)$. 
Lemma~\ref{normalising} gives commutative diagrams
\begin{equation*}\label{eq:CD1}
\begin{tikzpicture}[>=stealth, yscale=0.6]
\node (P2) at (4,3) {$M(\ESS^d)$};
\node (Pm+1) at (4,0) {$M_{m+1}^{\sub}(\ESS^d)$};
\node (P1) at (0,3) {$M(\ESS^d)$};
\node (Pm) at (0,0) {$M_m^{\sub}(\ESS^d)$};
\node (mum) at (-2,3) {$\mu_m$};
\node (num) at (-2,0) {$c_m\nu_{\mu_m}$};
\node (mum1) at (6.5,3) {$\mu_{m+1}$};
\node (num1) at (6.5,0) {$c_{m+1}\nu_{\mu_{m+1}}$,};
\draw[->] (Pm+1)--(Pm) node[above, midway] {$E^T_{m*}$};
\draw[->] (P2)--(P1) node[above, midway] {$E^T_{m*}$};
\draw[|->] (mum)--(num);
\draw[|->] (mum1)--(num1);
\draw[->] (P1)--(Pm);
\draw[->] (P2)--(Pm+1);
\end{tikzpicture}\end{equation*}
and Theorem~\ref{thm:sub-inv} implies that the vertical arrows are bijections. A simple set-theoretic argument then implies that we also
have commutative diagrams
\begin{equation*}\label{eq:CD2}
\begin{tikzpicture}[>=stealth, yscale=0.6]
\node (P2) at (4,3) {$P(\ESS^d)$};
\node (Pm+1) at (4,0) {$P_{m+1}^{\sub}(\ESS^d)$};
\node (P1) at (0,3) {$P(\ESS^d)$};
\node (Pm) at (0,0) {$P_m^{\sub}(\ESS^d)$};
\node (mum) at (-2,3) {$c_m^{-1}\mu_{\nu_m}$};
\node (num) at (-2,0) {$\nu_m$};
\node (mum1) at (6.5,3) {$c_{m+1}^{-1}\mu_{\nu_{m+1}}$};
\node (num1) at (6.5,0) {$\nu_{m+1}$.};
\draw[->] (Pm+1)--(Pm) node[above, midway] {$E^T_{m*}$};
\draw[->] (P2)--(P1) node[above, midway] {$E^T_{m*}$};
\draw[|->] (num)--(mum);
\draw[|->] (num1)--(mum1);
\draw[->] (Pm)--(P1);
\draw[->] (Pm+1)--(P2);
\end{tikzpicture}\end{equation*}

Thus the sequence $(\mu_m):=(c_m^{-1}\mu_{\nu_m})$ belongs to the inverse limit
$\varprojlim\big(P(\ESS^d),E^T_{m*}\big)$, and hence is given by a probability measure
$\mu\in P\big(\varprojlim(\ESS^d,E^T_m)\big)$.  We want to show that $\phi=\psi_\mu$.
Since both are states, it suffices to check that they agree on elements
$V_{m,p}U_{m,n}V_{m,q}^*$. Since $\phi$ is a KMS$_\beta$ state and the measure $\nu_m$
implements $\phi$ on $C(\ESS^d)=\clsp\{U_{m,n}\}$, we have
\[
\phi\big(V_{m,p}U_{m,n}V_{m,q}^*\big)=\delta_{p,q}e^{-\beta p^Tr^m}\int_{\ESS^d} e^{2\pi ix^Tn}\,d\nu_m(x).
\]
Since $\nu=\nu_{\mu_\nu}$ for all $\nu$ and $\mu_{\nu_m}=c_m\mu_m$, we have
$\nu_m=c_m\nu_{\mu_m}$ and
\[
\phi\big(V_{m,p}U_{m,n}V_{m,q}^*\big)=\delta_{p,q}e^{-\beta p^Tr^m}c_m\int_{\ESS^d} e^{2\pi i x^Tn}\,d\nu_{\mu_m}(x).
\]
This is precisely the formula for $\psi_\mu\big(V_{m,p}U_{m,n}V_{m,q}^*\big)$
in~\eqref{defpsimu}. Thus $\phi=\psi_\mu$.

Since each $\psi_\mu$ is a state, it follows from the formula~\eqref{defpsimu} that
$\mu\mapsto\psi_\mu$ is affine and weak* continuous from $M(\ESS^d)=C(\ESS^d)^*$ to the
state space of $B_\infty$. The formula~\eqref{defpsimu} also implies that $\mu\mapsto
\psi_\mu$ is injective, and since we have just shown that it is surjective, we deduce
that it is a homeomorphism of the compact space $P\big(\varprojlim(\ESS^d,E^T_n)\big)$
onto the simplex of KMS$_\beta$ states of $(B_\infty,\alpha)$.
\end{proof}

\begin{proof}[Proof of Theorem~\ref{newbigthm}]
According to~\eqref{defpsimu} in Theorem~\ref{bigthm}, we have to compute
\[
\int_{[0,\infty)^k} e^{2\pi ix^Tn}\,d\nu_{\mu_m}(x),
\]
which by Theorem~\ref{thm:sub-inv} is
\begin{align}\label{Laplace}
\int_{[0,\infty)^k}e^{-\beta w^Tr^m}&\int_{\ESS^d} e^{2\pi i(x+\theta_m^Tw)^Tn}\,d\mu_m(x)dw\\
&=\int_{[0,\infty)^k}e^{-\beta w^Tr^m}e^{2\pi iw^T\theta_m n}\int_{\ESS^d} e^{2\pi ix^Tn}\,d\mu_m(x)dw\notag\\
&=\int_{[0,\infty)^k}e^{-\beta w^Tr^m}e^{2\pi iw^T\theta_m n}M_{m,n}(\mu)\,dw.\notag
\end{align}
We can rewrite the integrand as
\[
e^{-\beta w^Tr^m}e^{2\pi iw^T\theta_m n}
=e^{\sum_{j=1}^kw_j(-\beta r^m_j+2\pi i(\theta_mn)_j)}=\prod_{j=1}^k e^{w_j(-\beta r^m_j+2\pi i(\theta_mn)_j)}.
\]
When we view $\int_{[0,\infty)^k}\,dw$ as an iterated integral, we find that
\[
\eqref{Laplace}=\prod_{j=1}^k\Big(\int_0^\infty e^{w_j(-\beta r^m_j+2\pi i(\theta_mn)_j)}M_{m,n}(\mu)\,dw_j\Big).
\]
Since $\beta>0$ and each $r^m_j>0$, we have
\[
\big|e^{w_j(-\beta r^m_j+2\pi i(\theta_mn)_j)}\big|=e^{w_j(-\beta r^m_j)}\to 0\quad\text{as $w_j\to \infty$}.
\]
Thus
\begin{align*}
\eqref{Laplace}=\prod_{j=1}^k\frac{e^{w_j(-\beta r^m_j+2\pi i(\theta_mn)_j)}}{-\beta r^m_j+2\pi i(\theta_mn)_j}M_{m,n}(\mu)\bigg|^\infty_0
&=\prod_{j=1}^k\frac{1}{\beta r^m_j-2\pi i(\theta_mn)_j}M_{m,n}(\mu),
\end{align*}
and the result follows from~\eqref{defpsimu}.
\end{proof}

\end{document}